\documentclass[11pt,a4paper]{amsart}

\usepackage{hyperref}
\usepackage{amssymb}
\usepackage{amsmath}
\usepackage{amsfonts}
\usepackage[alphabetic]{amsrefs}
\usepackage{ascmac}
\usepackage{bm}
\usepackage{stmaryrd}
\usepackage{graphicx}
\usepackage{pgf,pgfplots}
\pgfplotsset{compat=1.16}
\usepackage{epsfig}
\usepackage{tikz}
\usepackage{tikz-cd}
\usepackage{mathrsfs}
\usepackage{enumerate}

\usepackage{amsthm}
\theoremstyle{plain}
\newtheorem{theo}{Theorem}[section]
\newtheorem{lemm}[theo]{Lemma}
\newtheorem{prop}[theo]{Proposition}
\newtheorem{coro}[theo]{Corollary}
\theoremstyle{definition}

\newtheorem{defi}[theo]{Definition} 
\theoremstyle{remark}
\newtheorem{exam}[theo]{Example}
\newtheorem{rema}[theo]{Remark}

\DeclareMathOperator{\ISA}{\mathbf{ISA}}
\DeclareMathOperator{\IS}{\mathbf{IS}}
\DeclareMathOperator{\EG}{\mathbf{EG}}
\DeclareMathOperator{\Calg}{\mathbf{C^*_{alg}}}
\DeclareMathOperator{\ComCalg}{\mathbf{ComC^*_{alg}}}
\DeclareMathOperator{\LCHS}{\mathbf{LCH}}
\DeclareMathOperator{\CHS}{\mathbf{CH}}
\DeclareMathOperator{\ISAlcH}{\mathbf{ISA_{lcH}}}
\DeclareMathOperator{\EGlcH}{\mathbf{EG_{lcH}}}

\DeclareMathOperator{\SP}{SP}
\DeclareMathOperator{\TG}{TG}
\DeclareMathOperator{\SA}{SA}
\DeclareMathOperator{\Bis}{Bis}
\DeclareMathOperator{\Gu}{Gu}

\renewcommand{\:}{\colon}
\newcommand{\inv}{^{-1}}
\DeclareMathOperator{\id}{id}

\newcommand{\0}{^{(0)}}
\newcommand{\2}{^{(2)}}

\newcommand{\op}{^{\mathrm{op}}}

\usepackage{xspace}
\newcommand{\egmor}{couple morphism\xspace}
\newcommand{\egmors}{couple morphisms\xspace}

\newcommand{\Egmors}{Couple morphisms\xspace}
\newcommand{\isamor}{action morphism\xspace}
\newcommand{\isamors}{action morphisms\xspace}

\newcommand{\shom}{$*$-ho\-mo\-mor\-phism\xspace}
\newcommand{\shoms}{$*$-ho\-mo\-mor\-phisms\xspace}

\title[A categorical approach to semigroups, groupoids, C*-algebras]{A categorical approach to inverse semigroups, \'etale groupoids, and their C*-algebras}
\author{Takuto Fujieda \and Takeshi Katsura \and Tomoki Uchimura}

\begin{document}

\maketitle

\begin{abstract}
We introduce a category of inverse semigroup actions and a category of \'etale groupoids. 
We show that there are three functors which send inverse semigroups to their spectral actions, inverse semigroup actions to their transformation groupoids, and \'etale groupoids to their groupoid C*-algebras, respectively. 
The composition of these functors is naturally equivalent to the already known functor which sends inverse semigroups to their C*-algebras. 
This result is a categorical version of the theorem by Paterson which states that the inverse semigroup C*-algebras coincide the groupoid C*-algebras of the universal groupoids. 
We also construct a functor from the category of \'etale groupoids to the category of inverse semigroup actions sending \'etale groupoids to their slice actions, 
and show that this functor is right adjoint to the functor sending inverse semigroup actions to their transformation groupoids. 
\end{abstract}


\section{Introduction}

In the $C^*$-algebra theory, the construction of $C^*$-algebras from inverse semigroups and \'etale groupoids are studied well.
Inverse semigroups and semigroup homomorphisms form the category $\IS$, and $C^*$-algebras and \shoms form the category $\Calg$.
It is easy to see that the construction of inverse semigroup $C^*$-algebras $C^*(S)$ from inverse semigroup $S$ 
becomes a functor $C^*$ from $\IS$ to $\Calg$. 
\'Etale groupoids and continuous groupoid homomorphisms also form a category. 
However, as explained below, the construction of groupoid $C^*$-algebras $C^*(G)$ from \'etale groupoids $G$ 
is far from being a functor with respect to this category. 
In this paper, we propose a category $\EG$ whose objects are \'etale groupoids 
with which the construction of groupoid $C^*$-algebras becomes a functor. 

From a group action, or more generally from an inverse semigroup action, 
one can construct an \'etale groupoid called a transformation groupoid. 
From an inverse semigroup $S$, one gets an inverse semigroup action $\big(S,\widehat{E(S)},\beta^S\big)$ called a spectral action. 
The transformation groupoid of the spectral action of $S$ is called a universal groupoid of $S$, and denoted by $\Gu(S)$. 
Paterson showed that for every inverse semigroup $S$, the inverse semigroup $C^*$-algebra $C^*(S)$ is isomorphic to the groupoid $C^*$-algebra $C^*(\Gu(S))$ of the universal groupoid $\Gu(S)$ of $S$ (see \cite[Theorem 4.4.1]{Pat99});
\[
\begin{tikzcd}
S&&&C^*(S).
\\
&\big(S,\widehat{E(S)},\beta^S\big)
&\Gu(S)
&
\arrow[from = 1-1, to = 1-4, mapsto, sloped, ""{name = a}]
\arrow[from = 1-1, to = 2-2, mapsto, sloped]
\arrow[from = 2-2, to = 2-3, mapsto, sloped, ""{name = b}]
\arrow[from = 2-3, to = 1-4, mapsto, sloped]
\end{tikzcd}
\]
It is not obvious how to define morphisms between inverse semigroup actions. 
In this paper, we define a category $\ISA$ whose objects are inverse semigroup actions, 
and show that constructions of spectral actions and transformation groupoids become 
functors $\SP$ from $\IS$ to $\ISA$ and $\TG$ from $\ISA$ to $\EG$, respectively. 
With the categories and the functors above, 
we get a categorical version of the theorem by Paterson. 
We also define a functor $\SA$ from $\EG$ to $\ISA$ 
which sends \'etale groupoids to their slice actions, 
and show that the functor $\TG$ is left adjoint to the functor $\SA$. 

We are going to explain the details. 
Discrete groups and compact spaces are particular examples of \'etale groupoids, 
and the construction of groupoid $C^*$-algebras from \'etale groupoids 
is extension of the one of group $C^*$-algebras from discrete groups 
and the one of function algebras from compact spaces. 
Continuous groupoid homomorphisms between discrete groups and between compact spaces 
are, of course, group homomorphisms and continuous maps, respectively. 
These induce \shoms between group $C^*$-algebras and between function algebras. 
However a group homomorphism produces a \shom between group $C^*$-algebras covariantly, 
whereas a continuous map produces a \shom between function algebras contravariantly. 
In general, a continuous groupoid homomorphism between \'etale groupoids 
does not induce a natural \shom between groupoid $C^*$-algebras 
covariantly or contravariantly. 
To overcome this difficulty, we introduce morphisms between \'etale groupoids called \emph{\egmors}.
These are couples of two continuous groupoid homomorphisms in two certain classes.
We show that all \'etale groupoids and \egmors form the category $\EG$ in Theorem \ref{Theorem: category EG}.
When constructing $C^*$-algebras and \shoms, 
it is natural to assume that \'etale groupoids have locally compact Hausdorff unit spaces, and that \egmors have a certain properness.
We see that all \'etale groupoids with locally compact Hausdorff units and all proper \egmors form the subcategory $\EGlcH$ of $\EG$ 
in Theorem \ref{Theorem: category EGlcH}.
We will construct a functor $C^*$ from $\EGlcH$ to $\Calg$ which sends \'etale groupoids to their (full) groupoid $C^*$-algebras 
in Theorem \ref{Theorem: functor C*} 
by defining natural \shoms from proper \egmors. 
We note that between locally compact spaces considered as \'etale groupoids, 
proper \egmors are nothing but proper partial maps in the opposite direction. 
These are morphisms of the opposite category $\LCHS\op$ of 
the category $\LCHS$ of locally compact spaces introduced in Subsection~\ref{Subsection: Partial maps}. 

An inverse semigroup action is a triplet $(S,X,\alpha)$ of an inverse semigroup $S$, a topological space $X$ 
and an action $\alpha$ of $S$ on $X$. 
From the observations discussed above, 
a natural candidate of a morphism between inverse semigroup actions 
is a pair of a semigroup homomorphism 
and a partial map between spaces in the opposite direction 
which are compatible with actions. 
This is called an \emph{\isamor} in Definition~\ref{Definition: isamor}. 
We show that all inverse semigroup actions and all \isamors form the category $\ISA$ in Theorem \ref{Theorem: category ISA}.
In this paper, we show several results which suggest that the definitions of the categories $\ISA$ and $\EG$ 
are the desired ones. 
One such result says that the construction of transformation groupoids $S\ltimes_\alpha X$ 
from inverse semigroup actions $(S,X,\alpha)$ becomes a functor $\TG$ from $\ISA$ to $\EG$ 
(Theorem \ref{Theorem: functor TG}). 
We get this result by showing that \isamors between inverse semigroup actions 
naturally induce \egmors between their transformation groupoids. 
Another result says that the construction of slice actions $\big(\Bis G, G\0,\gamma^G\big)$ 
from \'etale groupoids $G$ becomes a functor $\SA$ from $\EG$ to $\ISA$
(Theorem \ref{Theorem: SA functoriality}). 
We note that general continuous groupoid homomorphisms do not induce 
semigroup homomorphisms between inverse semigroups $\Bis G$ 
consisting of open bisections of \'etale groupoids $G$ in any direction, 
whereas \egmors between \'etale groupoids 
naturally induce \isamors between their slice actions 
in the right direction. 
Moreover we show that the functor $\TG$ is left adjoint to the functor $\SA$ 
in Theorem \ref{Theorem: adj func btwn ISA and EG};
\[
\begin{tikzcd}
\ISA & \EG.
\arrow[r, from = 1-1, to = 1-2, "\TG", ""'{name = a}, bend left = 20]
\arrow[r, from = 1-2, to = 1-1, "\SA", ""'{name = b}, bend left = 20]
\arrow[phantom, from = a, to = b, "\perp"]
\end{tikzcd}
\]
Similarly as the subcategory $\EGlcH$ of $\EG$, 
all inverse semigroup actions $(S,X,\alpha)$ on locally compact Hausdorff spaces $X$ 
and \isamors $(\theta,\xi)$ with proper partial maps $\xi$ 
define a subcategory $\ISAlcH$ of $\ISA$ (Theorem \ref{Theorem: category ISAlcH}). 
We show that the two functors $\SA$ and $\TG$ descend functors between $\EGlcH$ and $\ISAlcH$, 
and these are also adjoint to each other. 

One more result suggesting that the definition of the category $\ISA$ 
(or its subcategory $\ISAlcH$) is the desired one 
is that the construction of spectral actions defines a functor from $\IS$ to $\ISAlcH$. 
From a semigroup homomorphism $\theta \: S \rightarrow T$ between inverse semigroups $S,T$, 
one naturally get a proper partial map $\widehat{\theta} \: \widehat{E(T)} \rightarrow \widehat{E(S)}$ 
between locally compact Hausdorff spaces. 
The domain of this map $\widehat{\theta}$ is not the whole space $\widehat{E(T)}$ in general. 
We show that the pair $\big(\theta,\widehat{\theta}\,\big)$ becomes a proper \isamor 
from the spectral action $\big( S,\widehat{E(S)},\beta^S \big)$ to $\big( T,\widehat{E(T)},\beta^T \big)$
in Proposition~\ref{Proposition: proper morphism}. 
This defines a functor from $\IS$ to $\ISAlcH$ which we denote $\SP$ (Theorem \ref{Theorem: functor SP}).

Now we are ready to state Theorem \ref{Theorem: Paterson categorical ver} which is a categorical version of the theorem by Paterson. 
Namely, the composition of the three functors $\SP\: \IS \rightarrow \ISAlcH$, $\TG\: \ISAlcH \rightarrow \EGlcH$, and $C^*\: \EGlcH \rightarrow \Calg$ are naturally equivalent to the functor $C^*\: \IS \rightarrow \Calg$;
\[
\begin{tikzcd}
\IS&&&\Calg.\\
&\ISAlcH&\EGlcH&
\arrow[from = 1-1, to = 1-4, rightarrow, ""'{name = a}, "C^*"]
\arrow[from = 1-1, to = 2-2, rightarrow, "\SP"']
\arrow[from = 2-2, to = 2-3, rightarrow, ""{name = b}, "\TG"']
\arrow[from = 2-3, to = 1-4, rightarrow, "C^*"']
\arrow[from = a, to = b, "\simeq", phantom, sloped]
\end{tikzcd}
\]
In our sequel \cite{FKU25}, we use these categories and functors to study quotient objects of 
inverse semigroups, inverse semigroup actions, \'etale groupoids and $C^*$-algebras. 

This paper is organized as follows. 
In Section \ref{Section: Preliminaries}, 
we introduce partial maps and recall definitions of inverse semigroups, 
inverse semigroup actions, \'etale groupoids and $C^*$-algebras. 
We also recall the relationships between them, and the theorem by Paterson mentioned above. 
In Section \ref{Section: A category ISA and a functor from IS to ISA}, 
we define the category $\ISA$ of inverse semigroup actions and \isamors, 
and the functor $\SP$ from $\IS$ to $\ISA$ of constructing spectral actions. 
The category $\EG$ of \'etale groupoids and \egmors is defined in Section \ref{Section: A category EG}. 
In Section \ref{Section: functors TG and SA}, 
we define the functor $\TG$ from $\ISA$ to $\EG$ of constructing transformation groupoids, 
and the functor $\SA$ from $\EG$ to $\ISA$ of constructing slice actions. 
We also show that $\TG$ is left adjoint to $\SA$. 
In Section \ref{Section: A functor C* from EG to Calg}, 
we define the functor $C^*$ from $\EGlcH$ to $\Calg$, 
and show our categorical version of the theorem by Paterson. 

Many papers including \cite{BEM12} consider categories of actions 
of a fixed inverse semigroup. 
However to the best of our knowledge, 
this paper is the first one which consider a category of inverse semigroup actions 
whose objects are triplets of inverse semigroups, spaces and actions. 
On the other hands, several previous researches consider categories of groupoids, 
and functors to the category $\Calg$ of $C^*$-algebras. 
For example, algebraic morphisms studied by Buneci and Stachura in \cite{BS06} 
(see also \cite{Zak90,Bun08}) 
define a category of locally compact, $\sigma$-compact, Hausdorff groupoids with Haar systems 
together with a functor to $\Calg$. 
In this paper, we only consider \'etale groupoids, 
but we do not assume that groupoids are $\sigma$-compact nor Hausdorff. 
We note that it is important to consider non-Hausdorff groupoids 
(see \cite{BEM12,Tay23,Kom21,FKU25}). 
As one can see from the explanation on actors below, 
the set of algebraic morphisms in \cite{BS06} 
can be considered as a proper subset of the set of \egmors introduced in this paper 
for locally compact, $\sigma$-compact, Hausdorff \'etale groupoids. 
In \cite{BEM12}, Buss, Exel and Meyer use algebraic morphisms between \'etale groupoids 
to study the relationship between the category of inverse semigroups and the one of \'etale groupoids. 
Theorem \ref{Theorem: adj func btwn ISA and EG} in this paper relates to 
\cite[Theorem 4.11]{BEM12}. 
One can obtain and extend \cite[Theorem 4.11]{BEM12} by combining Theorem \ref{Theorem: adj func btwn ISA and EG} in this paper 
with one results on adjoint functors between the non-Hausdorff version of 
the functor $\SP$ from $\IS$ to $\ISA$ in Theorem \ref{Theorem: functor SP} 
and the forgetful functor from $\ISA$ to $\IS$. 
This theme is pursued elsewhere. 
We note that in \cite{BEM12} inverse semigroups are assumed to have zero and unit. 
This is because the set of algebraic morphisms are ``smaller'' than the one of \egmors. 
One can get rid of this assumption using \egmors instead of algebraic morphisms. 
This is one advantage of \egmors introduced in this paper. 
Lawson and Lenz also considered similar adjunction in \cite[Theorem 2.22]{LL13}. 
The objects and the morphisms of the categories $\mathbf{Inv}$ and $\mathbf{Etale}$ in \cite{LL13}
seem to be smaller than the ones of $\IS$ and $\EG$ in this paper. 
While this paper has been prepared, \cite{Tay23} was posted on the ArXiv. 
In \cite{Tay23}, Taylor uses actors to define a category of \'etale groupoids. 
Actors are introduced by Meyer and Zhu in \cite{MZ15} 
which can be considered as \'etale version of algebraic morphisms in \cite{BS06}. 
In \cite[Lemma 6.2]{Tay23}, Taylor constructs an \'etale groupoid 
and two groupoid homomorphisms from actors satisfying some conditions 
(see also \cite[Lemma 4.3]{BEM12}). 
These define a \egmor introduced in this paper.  
This construction can be generalized to the construction from arbitrary actors, 
but not all \egmors can be obtained from actors. 
In the sequel of this paper, 
we consider groupoid correspondences between \'etale groupoids. 
There it will be shown that the sets of actors and the ones of \egmors can be considered 
as subsets of the set of groupoid correspondences. 
Under this observation, the set of actors are strictly smaller than the one of \egmors. 
We borrow the notion ``fibrewise bijective'' from \cite[Definition~6.1]{Tay23} 
to define \egmors.


\section{Preliminaries}
\label{Section: Preliminaries}

\subsection{Partial maps}
\label{Subsection: Partial maps}

For topological spaces $X$ and $Y$, a \emph{partial map} from $X$ to $Y$ is a continuous map from an open subspace $D_f$ of $X$ to $Y$. 
This is denoted as $f\:X\supset D_f\rightarrow Y$ or just $f\: X \rightarrow Y$. 
We call $D_f$ the domain of $f$.
For a subset $V$ of $Y$, we define a subset $f\inv(V)$ of $X$ as the set of all elements $x \in D_f$ with $f(x)\in V$.
For partial maps $f_1\: X\supset D_{f_1}\rightarrow Y$ and $f_2\: Y\supset D_{f_2}\rightarrow Z$, the composition $f_2\circ f_1\: X \supset D_{f_2\circ f_1} \rightarrow Z$ is defined as $f_2\circ f_1(x) = f_2(f_1(x))$ for every $x \in D_{f_2\circ f_1} := f_1\inv(D_{f_2})$.
This composition satisfies the associative law.
We remark that $(f_2 \circ f_1)\inv(W) = f_1\inv\big(f_2\inv(W)\big)$ holds for a subset $W$ of $Z$.
A partial map $f\:X\supset D_f\rightarrow Y$ is said to be \emph{proper} if for every compact subspace $K$ of $Y$, the subspace $f\inv(K)$ of $X$ is compact.
It is clear that the composition of proper partial maps is proper.
We denote $\LCHS$ as the category of locally compact Hausdorff spaces and proper partial maps.

We denote $\CHS_*$ as the category of compact Hausdorff spaces with fixed points and continuous maps preserving fixed points.
For a compact Hausdorff space $X$ with a fixed point $x_0\in X$, the open subspace $X\setminus\{x_0\}$ is a locally compact Hausdorff space.
For compact Hausdorff spaces $X$, $Y$ with fixed points $x_0\in X$, $y_0\in Y$ and a continuous map $f\: X\rightarrow Y$ with $f(x_0) = y_0$, the restriction of $f$ to the open subspace $f\inv\left(Y\setminus\{y_0\}\right)$ of $X\setminus\{x_0\}$ is a proper continuous map.
Thus $f$ induces a proper partial map from $X\setminus\{x_0\}$ to $Y\setminus\{y_0\}$.
These constructions become a functor from $\CHS_*$ to $\LCHS$.

For a locally compact Hausdorff space $X$, we can define the compact topology on the union set $\widetilde{X}$ of $X$ and another point $\infty_X$ called the \emph{one-point compactification} of $X$.
For locally compact Hausdorff spaces $X,Y$ and a proper partial map $f\: X \supset D_f \rightarrow Y$, we define a map $\widetilde{f}\: \widetilde{X} \rightarrow \widetilde{Y}$ as a map sending an element $x\in D_f$ to $f(x) \in Y$ and elements of $\widetilde{X} \setminus D_f$ to $\infty_Y$.
This map $\widetilde{f}$ is continuous.
These constructions become a functor from $\LCHS$ to $\CHS_*$.
The above two functors are equivalences between $\LCHS$ and $\CHS_*$.
Thus $\LCHS$ and $\CHS_*$ are equivalent categories.

\subsection{$C^*$-algebras}

We recall the motivating fact in the $C^*$-algebra theory.
See \cite{Mur90,Dav96} for detail of the $C^*$-algebra theory.

A $*$-algebra is a $\mathbb{C}$-algebra with involution.
A $C^*$-algebra is a $*$-algebra which has a submultiplicative complete norm satisfying the $C^*$-condition.
A \shom between $*$-algebras is a linear map which keeps multiplication and involution.
A \shom between $C^*$-algebras becomes norm-decreasing automatically.
\begin{defi}
	We denote $\Calg$ as the category of all $C^*$-algebras and all \shoms.
\end{defi}
A locally compact Hausdorff space $X$ produces a commutative $C^*$-algebra $C_0(X)$ of continuous complex functions on $X$ vanishing at infinity.
We denote $\ComCalg$ as the full subcategory of $\Calg$ consisting of all commutative $C^*$-algebras.
The construction $X \mapsto C_0(X)$ forms a contravariant functor from $\LCHS$ to $\ComCalg$.
The Gelfand-Naimark theorem says that this functor is an equivalence between $\LCHS\op$ and $\ComCalg$, where $\LCHS\op$ is the opposite category of $\LCHS$.
This fact is one of our motivation for introducing \egmor and \isamor defined later.
This fact is also mentioned in \cite{AM21}.

\subsection{Inverse semigroups and their $C^*$-algebras}
\label{Subsection: Inverse semigroups and their C*-algebras}

In this subsection, we recall the basics of inverse semigroups and their $C^*$-algebras. 
See \cite[Chapter 2]{Pat99}, \cite[Section 2]{BEM12}, or \cite[Section 4]{Exe08} for more detail.

\subsubsection{Inverse semigroups and their spectral actions}

\begin{defi}
	An \emph{inverse semigroup} $S$ is a semigroup such that for every element $s\in S$ there exists a unique element $s^*\in S$ with $ss^*s = s$ and $s^*ss^* = s^*$.
\end{defi}

We call $s^*$ as the \emph{generalized inverse} of $s$.
An element $e\in S$ which satisfies $e^2=e$ is called an \emph{idempotent}. 
Every idempotent $e$ satisfies $e^* =e$.
For every element $s\in S$, $s^*s$ and $ss^*$ are idempotents. 
We denote the set of all idempotents as $E(S)$. 
It is well-known that $E(S)$ becomes a commutative subsemigroup of $S$.

\begin{exam}
	A discrete group is an inverse semigroup which has the unit as a unique idempotent.
\end{exam}

\begin{defi}
	A \emph{semigroup homomorphism} between inverse semigroups is a map which keeps multiplications.
\end{defi}

A semigroup homomorphism keeps generalized inverses automatically.

\begin{defi}
	We denote $\IS$ as the category of all inverse semigroups and all semigroup homomorphisms.
\end{defi}

We give an important example of inverse semigroup.
Let $X$ be a topological space.
A \emph{partial homeomorphism} on $X$ is a homeomorphism between open subspaces of $X$.
The set of all partial homeomorphisms on $X$ becomes an inverse semigroup with respect to the composition as partial maps.

\begin{defi}
	We denote $I(X)$ as the inverse semigroup of all partial homeomorphisms on $X$.
\end{defi}

We recall the basics of inverse semigroup actions.

\begin{defi}
	An \emph{inverse semigroup action} $(S,X,\alpha)$ is a triplet of an inverse semigroup $S$, a topological space $X$, and a semigroup homomorphism $\alpha\: S \rightarrow I(X); s \mapsto \alpha_s$ such that the union of all domains of $\alpha_s$ coincides with $X$.
\end{defi}

\begin{rema}
In \cite[Definition 5.1]{Ste10}, a semigroup homomorphism $\alpha$ such that the union of all domains of $\alpha_s$ coincides with $X$ 
is called a non-degenerate action. 
\end{rema}

For every $s\in S$, we denote the domain of $\alpha_s$ as $D^\alpha_s$.
We can check that $D^\alpha_s$ coincide with $D^\alpha_{s^*s}$.

Let $S$ be an inverse semigroup. 
We will define an inverse semigroup action called the spectral action of $S$.
We first remark that $\{0,1\}$ becomes an inverse semigroup with the usual multiplication. 
We denote $\widehat{E(S)}_0$ as the set of all semigroup homomorphism from $E(S)$ to $\{0,1\}$.
Give the Hausdorff topology to $\{0,1\}$. 
We regard $\widehat{E(S)}_0$ as the closed subspace of the product space $\{0,1\}^{E(S)}$ which is compact Hausdorff.
We write $0$ as the semigroup homomorphism which sends all elements of $E(S)$ to $0$.
We set $\widehat{E(S)} := \widehat{E(S)}_0\setminus\{0\}$.
This is a locally compact Hausdorff space.
An element of $\widehat{E(S)}$ is called a \emph{character} on $E(S)$.
We set a subspace $U_e^S$ of $\widehat{E(S)}$ as 
\[
U_e^S := \big\{ \zeta \in \widehat{E(S)} \ \big|\ \zeta(e)=1 \big\}
\] 
for every $e\in E(S)$. 
The subspace $U_e^S$ becomes a compact open subspace of $\widehat{E(S)}$ for all $e \in E(S)$.
We can easily check that 
$
\bigcup_{e \in E(S)} U_e^S = \widehat{E(S)}.
$

\begin{defi}
	For an inverse semigroup $S$, we define \emph{the spectral action} $\big(S,\widehat{E(S)}, \beta^S\big)$ of $S$ as follows:
	For every $s\in S$, we set a subspace $U_s^S$ of $\widehat{E(S)}$ as $U_{s^*s}^S$. 
	For every $s\in S$, we define a partial homeomorphism $\beta_s^S\: U_{s}^S \rightarrow U_{s*}^S$ of $\widehat{E(S)}$ as $\beta_s^S(\zeta)(e) = \zeta(s^*es)$ for $\zeta\in U^S_s = U_{s^*s}^S$ and $e\in E(S)$. 
\end{defi}
In Section \ref{Section: A category ISA and a functor from IS to ISA}, we will introduce morphisms between inverse semigroup actions, and show that the construction of spectral actions $S \mapsto \big(S,\widehat{E(S)},\beta^S\big)$ becomes a functor.

\subsubsection{Inverse semigroup $C^*$-algebras}

Let $S$ be an inverse semigroup. 
We set $\mathbb{C}[S]$ as the set of all maps $f$ from $S$ to $\mathbb{C}$ such that $f(s) = 0$ for all but finite $s\in S$. 
This becomes a $*$-algebra with respect to the following multiplication and involution;
\[
\big(f_1*f_2\big)(s):= \sum_{\substack{s_1,s_2\in S\\s_1s_2=s}} f_1(s_1)f_2(s_2), \hspace{5mm} f^*(s) = \overline{f(s^*)}
\]
for $f_1,f_2,f\in \mathbb{C}[S]$ and $s\in S$. In other words, a $*$-algebra $\mathbb{C}[S]$ is generated by the delta functions $\left\{\delta_s\right\}_{s\in S}$ with respect to the multiplication defined by $\delta_s*\delta_t=\delta_{st}$ and the conjugate linear operation defined by $\left(\delta_s\right)^* = \delta_{s^*}$ for $s\in S$.

There exists a unique $C^*$-algebra $C^*(S)$ such that it contains $\mathbb{C}[S]$ as a dense $*$-subalgebra, and that any \shom from $\mathbb{C}[S]$ to any $C^*$-algebra $D$ extends to a unique \shom from it to $D$.
See \cite[p.26]{Pat99} for the existence of this $C^*$-algebra $C^*(S)$.
\begin{defi}
	We call the $C^*$-algebra $C^*(S)$ as the \emph{inverse semigroup $C^*$-algebra} of $S$.
\end{defi}

A semigroup homomorphism $\theta$ from $S$ to $T$ induces a \shom $\sigma_\theta$ from $\mathbb{C}[S]$ to $\mathbb{C}[T]$ by 
\[
\sigma_\theta(f)(t) := \sum_{\theta(s) = t} f(s)
\]
for $f\in\mathbb{C}[S]$ and $t\in T$. 
This map sends $\delta_s\in\mathbb{C}[S]$ to $\delta_{\theta(s)}\in\mathbb{C}[T]$. 
We denote the extension of the \shom $\sigma_\theta \: \mathbb{C}[S] \rightarrow \mathbb{C}[T] \subset C^*(T)$ to the \shom from $C^*(S)$ as the same symbol $\sigma_\theta$.
One can obtain easily the following proposition:

\begin{prop}
	The constructions $S \mapsto C^*(S)$ and $\theta\mapsto\sigma_\theta$ form a functor from $\IS$ to $\Calg$. 
	We denote the functor as $C^*$.
\end{prop}

\subsection{\'Etale groupoids and their $C^*$-algebras}

In this section, we recall the fundamentals of \'etale groupoids and groupoid $C^*$-algebras. 
See \cite[Section 3]{Exe08}, \cite[Section 3]{BEM12}, or \cite{KS02} for more detail.

\subsubsection{\'Etale groupoids and open bisections}

\begin{defi}
	A \emph{groupoid} consists of a set $G$ called an \emph{arrow space}, a subset $G\0 \subset G$ called a \emph{unit space}, maps $d,r\: G \rightarrow G\0$ called a \emph{domain map} and a \emph{range map} respectively, a map called a \emph{multiplication map}
	\[
	G^{(2)}:=\{ (g',g)\in G \times G \mid d(g') = r(g) \} \ni (g',g) \mapsto g'g \in G,
	\]
	and a map called an \emph{inverse map}
	\[
	G \ni g \mapsto g\inv \in G
	\]
	such that 
	\begin{enumerate}[(i)]
		\item $r(u) = u = d(u)$ for $u\in G\0$,
		\item $r(g)g = g = gd(g)$ for $g\in G$,
		\item $r(g'g) = r(g')$ and $d(g'g) = d(g)$ for $(g',g)\in G^{(2)}$,
		\item $g''(g'g) = (g''g')g$ for $(g'',g'),(g',g)\in G^{(2)}$,
		\item $d(g\inv) = r(g)$, $r(g\inv) = d(g)$ for $g \in G$,
		\item $gg\inv = r(g)$ and $g\inv g = d(g)$ for $g\in G$.
	\end{enumerate}
\end{defi}
A pair $g',g\in G$ is said to be \emph{composable} if $(g',g)\in G\2$.
We set $G_A:= d\inv(A)$, $G^B:=r\inv(B)$, and $G_A^B:=G_A\cap G^B$ for subsets $A,B\subset G\0$. 
For one point sets $A = \{x\}$ and $B = \{y\}$, we write them as $G_x$, $G^y$, and $G_x^y$. 
By (i), the domain and range maps are surjective.
For subsets $U$ and $V$ of $G$, we set
\[
UV:=\big\{ uv \ \big|\  u\in U, v\in V, (u,v)\in G\2 \big\}, \hspace{3mm} U\inv:=\big\{ u\inv \ \big|\  u\in U \big\}.
\]

\begin{defi}
	A subset $U$ of a groupoid $G$ is called a \emph{bisection} if both $d|_U$ and $r|_U$ are injective.
\end{defi}
For bisections $U$ and $V$ of $G$, $UV$ and $U\inv$ are bisections of $G$ (see \cite[Proposition 2.2.3]{Pat99})

\begin{defi}
	A \emph{topological groupoid} $G$ is a groupoid equipped with a topology on $G$ such that the multiplication map and the inverse map are continuous.
\end{defi}

For every topological groupoid $G$, the domain map and the range map are continuous by the condition (vi).
For every $g \in G$, we get $(g\inv)\inv = g$. 
Thus the inverse map $g \mapsto g\inv$ is a homeomorphism on $G$.

\begin{defi}
	A topological groupoid $G$ is said to be \emph{\'etale} if the domain map $d\: G \rightarrow G\0$ is locally homeomorphic, that is for every $g\in G$, there exists an open neighborhood $V_g\subset G$ of $g$ such that the restriction $d|_{V_g}$ is a homeomorphism onto an open subspace of $G\0$.
\end{defi}

We remark that the domain map $d$ of $G$ is locally homeomorphic if and only if so is the range map $r\: G\rightarrow G\0$. 
For an \'etale groupoid $G$ and a subset $A \subset G\0$, the topological subgroupoid $G_A^A$ is \'etale. 

\begin{prop}[{\cite[Proposition 3.2]{Exe08}}]\label{Proposition: G0 open}
	If $G$ is \'etale, the unit space $G\0$ is open in $G$.
\end{prop}
For every \'etale groupoid $G$, the subsets $G_u$ and $G^u$ are discrete for every $u\in G\0$.

\begin{exam}
	A discrete group is an \'etale groupoid whose unit space consists of the identity only.
\end{exam}

\begin{exam}
	A topological space $X$ is an \'etale groupoid whose unit space coincides with the whole space $X$.
\end{exam}

\begin{exam}
	For groupoids $G$ and $H$, the direct product $G\times H$ becomes a groupoid with respect to the structures defined component-wise. 
	For topological groupoids $G$ and $H$, the groupoid $G \times H$ becomes a topological groupoid with respect to the product topology.
	If $G$ and $H$ are \'etale, so is $G\times H$.
\end{exam}

Let $G$ be an \'etale groupoid.
In the studies of \'etale groupoids, open bisections play an important role. 
In \cite{Exe08}, open bisections are called \emph{slices}.
For every open bisections $U$ and $V$ of $G$, $UV$ and $U\inv$ are open bisections. 
All open bisections form an inverse semigroup with respect to the multiplication $(U,V) \mapsto UV$ (see \cite[Proposition 2.2.4]{Pat99}).
\begin{defi}
	We denote the inverse semigroup of all open bisections of $G$ as $\Bis G$.
\end{defi}

\begin{lemm}\label{Lemma: dr homeo on op bis}
	For every open bisection $U$ of $G$, the restrictions $d|_U\: U \rightarrow d(U)$ and $r|_U\: U \rightarrow r(U)$ of the domain and range maps to $U$ are homeomorphisms onto open subsets of $G\0$. 
\end{lemm}
\begin{proof}
	The maps $d|_U$ and $r|_U$ are continuous bijections.
	These are the restrictions of the open maps $d,r\: G \rightarrow G\0$ to an open subset $U$.
	Hence $d|_U$ and $r|_U$ are open.
\end{proof}

\begin{lemm}\label{Lemma: bisection}
	The set of all open bisections form a basis of $G$.
\end{lemm}
\begin{proof}
	For every element $g\in G$, there exist open subsets $U_g$ and $V_g$ such that $d|_{U_g}$ and $r|_{V_g}$ are injective. 
	For every open subset $U$ of $G$ with $g \in U$, $U\cap U_g \cap V_g$ is an open bisection containing $g$ and included in $U$.
\end{proof}

\begin{coro}\label{Corollary: bisection lcHs}
	Assume $G$ has a locally compact Hausdorff unit space.
	Every open bisection $U$ of $G$ is a locally compact Hausdorff space.
	The set of all open bisections becomes an open basis of $G$ consisting of Hausdorff subsets.
\end{coro}


\subsubsection{Groupoid homomorphisms}
Let $G,H$ be groupoids.
\begin{defi}
	A map $\varphi\: G \rightarrow H$ is called a \emph{groupoid homomorphism} if for every $(g',g)\in G\2$, $(\varphi(g'),\varphi(g))\in H\2$ and $\varphi(g'g) = \varphi(g')\varphi(g)$ hold. 
\end{defi}

Let $\varphi\: G \rightarrow H$ be a groupoid homomorphism.
We get $\varphi\big(G\0\big)\subset H\0$, and $\varphi\big(G_u)\subset H_{\varphi(u)}$ for every $u\in G\0$.
We denote the restriction of $\varphi$ to the unit space $G\0$ as $\varphi\0\:G\0\rightarrow H\0$ and the restriction to $G_u$ as $\varphi_u\:G_u\rightarrow H_{\varphi(u)}$ for every $u\in G\0$. 
By simple calculations, we get $\varphi\0 \circ d_G = d_H \circ \varphi$ and $\varphi\0 \circ r_G = r_H \circ \varphi$, where $d_G,r_G$ and $d_H,r_H$ are the domain and range maps of $G$ and $H$.
For $g \in G$, $\varphi(g\inv) = \varphi(g)\inv$ holds.

For a groupoid $G_i$ with = 1,2,3, and a groupoid homomorphism $\varphi_i\:G_i\rightarrow G_{i+1}$ with $i=1,2$, we get
\[
(\varphi_2\circ\varphi_1)\0 = \varphi_2\0 \circ \varphi_1\0, \hspace{3mm} (\varphi_2\circ\varphi_1)_u = (\varphi_2)_v\circ(\varphi_1)_u,
\]
where $u\in G_1\0$ and $v=\varphi_1(u)$.

\begin{lemm}\label{Lemma: grpd hom and inv}
	Let $\varphi\: G \rightarrow H$ be a groupoid homomorphism.
	The following hold:
	\begin{enumerate}[(i)]
		\item $\varphi(U\inv) = \varphi(U)\inv$ for a subset $U$ of $G$.
		\item $\varphi\inv(V\inv) = \varphi\inv(V)\inv$ for a subset $V$ of $H$.
	\end{enumerate}
\end{lemm}
\begin{proof}
	It follows from the fact that $\varphi(g\inv) = \varphi(g)\inv$ holds for $g \in G$.
\end{proof}

\begin{lemm}\label{Lemma: grpd hom inj surj}
	Let $\varphi\: G \rightarrow H$ be a groupoid homomorphism.
	\begin{enumerate}[(i)]
		\item $\varphi$ is injective if and only if $\varphi\0$ and $\varphi_u$ for all $u\in G\0$ are injective.
		\item $\varphi$ is surjective if $\varphi\0$ and $\varphi_u$ for all $u \in G\0$ are surjective.
		\item $\varphi\0$ is surjective if $\varphi$ is surjective.
		\item $\varphi_u$ is surjective for all $u \in G\0$ if $\varphi$ is surjective and $\varphi\0$ is injective.
		\item $\varphi$ is bijective if and only if $\varphi\0$ and $\varphi_u$ for all $u\in G\0$ are bijective.
	\end{enumerate}
\end{lemm}
\begin{proof}
	\begin{enumerate}[(i)]
		\item The only if part is clear.
		We show the if part.
		Take $g,g' \in G$ with $\varphi(g) = \varphi(g')$. 
		We get $\varphi\0(d(g))= d(\varphi(g)) = d(\varphi(g')) = \varphi\0(d(g'))$.
		Since $\varphi\0$ is injective, we have $d(g) = d(g') =:u$.
		Since $\varphi_u$ is injective, $\varphi_u(g) = \varphi(g) = \varphi(g') = \varphi_u(g')$ implies $g = g'$.
		Thus $\varphi$ is injective.
		\item Take $h \in H$.
		Since $\varphi\0$ is surjective, there exists $u \in G\0$ with $\varphi(u) = d(h)$.
		Since $\varphi_u\: G_u \rightarrow H_{d(h)}$ is surjective, there exists $g \in G_u$ with $\varphi(g) = h$.
		\item This follows from $\varphi\0 \circ d = d \circ \varphi$, and the fact that the domain maps $d$ are surjective.
		\item Take $u \in G\0$ and $h \in H_{\varphi\0(u)}$.
		Since $\varphi$ is surjective, there exists $g \in G$ with $\varphi(g) = h$.
		We get $\varphi\0(d(g)) = d(\varphi(g)) = d(h) = \varphi\0(u)$.
		Since $\varphi\0$ is injective, we have $d(g) = u$.
		Thus $g$ is an element of $G_u$.
		Hence $\varphi_u$ is surjective.
		\item This follows from (i)-(iv).\qedhere
	\end{enumerate}
\end{proof}

Let $G,H$ be \'etale groupoids.
\begin{lemm}\label{Lemma: phi open iff phi0 open}
	A continuous groupoid homomorphism $\varphi\: G\rightarrow H$ is open if and only if $\varphi\0\:G\0\rightarrow H\0$ is open.
\end{lemm}
\begin{proof}
	The only if part follows from the fact that $G\0$ is open in $G$ by Proposition \ref{Proposition: G0 open}.
	
	Assume that $\varphi\0$ is open.
	We check that $\varphi(U)$ is open for every open subset $U$ of $G$.
	It suffices to show that for every $g\in U$ there exists an open neighborhood $U_g\subset U$ of $g$ such that $\varphi(U_g)$ is open in $H$.
	
	For every element $g\in U$, there exists an open bisection $V_g$ of $H$ including $\varphi(g)$ by Lemma \ref{Lemma: bisection}.
	We set an open neighborhood $U_g$ of $g$ in $G$ as $U\cap \varphi\inv\left(V_g\right)$.
	The subset $\varphi(U_g)$ satisfies $\varphi(g) \in \varphi(U_g) \subset \varphi(U) \cap V_g$.
	Since $V_g$ is an open bisection in $H$ and Lemma \ref{Lemma: dr homeo on op bis} holds, the restriction $d|_{V_g}$ of the domain map $d$ is a homeomorphism to the open subset $d(V_g)$ in $H\0$.
	The subset $d(U_g)$ is open in $G\0$ because the domain map $d\: G \rightarrow G\0$ is open.
	By assumption, $d(\varphi(U_g)) = \varphi\0(d(U_g))$ is open in $H\0$, and also in $d(V_g)$.
	Thus $\varphi(U_g)$ is open in $V_g$, and also in $H$.
\end{proof}


\subsubsection{\'Etale groupoid $C^*$-algebras}

Let $G$ be an \'etale groupoid with a locally compact Hausdorff unit space, and $U$ be an open Hausdorff subset of $G$.
The open subset $U$ becomes locally compact.
We denote $C_c(U)$ as the set of all continuous complex functions $f$ on $U$ such that there exists a compact subset $K$ of $U$ with $f=0$ on $U\setminus K$. 
Every such functions can be extended to a function on $G$ by setting zero outside of $U$. 
We regard $C_c(U)$ as a subspace of the vector space of all functions on $G$. 
We define $Q(G)$ as the linear span of the union of the subspaces $C_c(U)$ for all Hausdorff open subsets $U$ of $G$. 
For every open Hausdorff basis $\{U_\lambda\}_{\lambda\in\Lambda}$ of $G$, the set $Q(G)$ coincides with the linear span of the union of $C_c(U_\lambda)$ for all $\lambda\in \Lambda$. 
In particular, the set $Q(G)$ coincides with the linear span of the union of $C_c(U)$ for all $U\in \Bis G$ by Corollary \ref{Corollary: bisection lcHs}.

The set $Q(G)$ becomes a $*$-algebra with respect to the following multiplication and involution;
\[
\big(f_1*f_2\big)(g) := \sum_{\substack{(g_1,g_2)\in G\2\\g_1g_2=g}} f_1(g_1)f_2(g_2), \hspace{3mm} f^*(g):=\overline{f(g\inv)}
\]
for $f_1,f_2,f\in Q(G)$ and $g\in G$. 
This fact follows from the next proposition:
\begin{prop}[{\cite[Proposition 3.11.]{Exe08}}]\label{Proposition: convolution}
	\begin{enumerate}[(i)]
		\item For $U_i\in \Bis G$ and $f_i \in C_c(U_i)$ with $i=1,2$, the convolution $f_1*f_2$ lies in $C_c(U_1U_2)$. 
		For $g_1\in U_1$ and $g_2\in U_2$ with $d(g_1) = r(g_2)$, we have
		\[
		(f_1*f_2)(g_1g_2) = f_1(g_1)f_2(g_2).
		\]
		\item For $U\in\Bis G$ and $f\in C_c(U)$, the involution $f^*$ lies in $C_c(U\inv)$.
	\end{enumerate}
\end{prop}

There exists a unique $C^*$-algebra $C^*(G)$ such that it contains $Q(G)$ as a dense $*$-subalgebra, and that any \shom from $Q(G)$ to any $C^*$-algebra $D$ extends to a unique \shom from it to $D$.
See \cite[p.204-205]{Exe08} for the existence of the $C^*$-algebra $C^*(G)$.

\begin{defi}
	We call the $C^*$-algebra $C^*(G)$ as the \emph{groupoid $C^*$-algebra} of $G$.
\end{defi}


In Section \ref{Section: A functor C* from EG to Calg}, we show that the construction of $C^*$-algebras $G \mapsto C^*(G)$ forms a functor from the category $\EGlcH$ of \'etale groupoids which we introduce in Subsection \ref{Subsection: A category EG of etale groupoids} to $\Calg$.

\subsection{Relationship between inverse semigroups and \'etale groupoids}

For an inverse semigroup action $(S,X,\alpha)$, we can construct an \'etale groupoid $S\ltimes_\alpha X$ as follows: 
The arrow space of $S\ltimes_\alpha X$ is defined as 
\[
S\ltimes_\alpha X := \{ (s,x)\in S\times X \mid x\in D^\alpha_s \}/{\sim},
\]
where $(s,x)\sim(t,y)$ if and only if $x=y$ and there exists $e\in E(S)$ with $x\in D^\alpha_e$ and $se = te$. 
We write the equivalence class of $(s,x)$ as $[s,x]$. 
The unit space is defined as 
\[
\big(S\ltimes_\alpha X\big)\0 := \{ [e,x]\in S\ltimes_\alpha X \mid e\in E(S)\}.
\]
The map 
\[
(S\ltimes_\alpha X)\0 \ni [e,x] \mapsto x \in X
\]
becomes a bijection.
We identify the unit space $(S\ltimes_\alpha X)\0$ as $X$ through the bijection above.
We set the domain map as $d\:[s,x]\mapsto x$ and the range map as $r\:[s,x]\mapsto \alpha_sx$. 
The multiplication is defined as $[t,y][s,x] := [ts,x]$ for every $([t,y],[s,x])\in (S\ltimes_\alpha X)\2$.
We define the inverse of $[s,x]$ as $[s^*,\alpha_sx]$ for every $[s,x]\in S\ltimes_\alpha X$.
The set $S\ltimes_\alpha X$ becomes a groupoid with respect to the above structures.
We give a topology of $S\ltimes_\alpha X$ generated by the collection of 
\[
[s,U] := \left\{[s,x] \in S\ltimes_\alpha X \ \middle|\  x\in U\right\}
\] 
for all pairs of $s\in S$ and an open subset $U$ of $D^\alpha_s$.
In this topology, the bijection $(S\ltimes_\alpha X)\0 \rightarrow X; [e,x]\mapsto x$ becomes a homeomorphism.
The domain and the range maps are local homeomorphisms.
Thus the groupoid $S\ltimes_\alpha X$ becomes an \'etale groupoid.

\begin{defi}
	For an inverse semigroup action $(S,X,\alpha)$, we call the \'etale groupoid $S\ltimes_\alpha X$ as the \emph{transformation groupoid} of $(S,X,\alpha)$.
\end{defi}

In Subsection \ref{Subsection: A functor TG from ISA to EG}, we show that the construction of transformation groupoids $(S,X,\alpha) \mapsto S\ltimes_\alpha X$ forms a functor between the categories $\ISA$ and $\EG$ which will be introduced in Section \ref{Section: A category ISA and a functor from IS to ISA} and \ref{Section: A category EG}, respectively.

\begin{defi}
	For an inverse semigroup $S$, we call the transformation groupoid $S\ltimes_{\beta^S} \widehat{E(S)}$ of the spectral action of $S$ as the \emph{universal groupoid} of $S$.
	We denote it as $\Gu(S)$.
\end{defi}

Recall that the domain of the partial homeomorphism $\beta_s^S$ of $S$ is denoted as $U^S_s$ for every $s\in S$.
We set the characteristic function of $\big[s,U^S_s\big] \subset \Gu(S)$ as $\chi_{[s,U^S_s]}$. 
Since $\big[s,U^S_s\big]$ is a compact open bisection, we get  $\chi_{[s,U^S_s]} \in Q(\Gu(S))$.
\begin{theo}\label{Theorem: Paterson}
	For every inverse semigroup $S$, the $C^*$-algebra $C^*(S)$ is isomorphic to the groupoid $C^*$-algebra $C^*(\Gu(S))$ through the $*$-isomorphism $\iota_S$ which sends $\delta_s\in C^*(S)$ to $\chi_{[s,U^S_s]}\in C^*(\Gu(S))$.
\end{theo}

See \cite[Theorem 4.1.1]{Pat99} or \cite[Theorem 3.3]{KS02} for a proof of Theorem \ref{Theorem: Paterson}.


\section{A category $\ISA$ and a functor $\SP$}
\label{Section: A category ISA and a functor from IS to ISA}

We define a morphism between inverse semigroup actions. 
Let $(S,X,\alpha)$ and $(T,Y,\beta)$ be inverse semigroup actions. 

\begin{defi}\label{Definition: isamor}
	An \emph{\isamor} $(\theta,\xi)$ from $(S,X,\alpha)$ to $(T,Y,\beta)$ consists of a semigroup homomorphism $\theta \: S\rightarrow T$ and a partial map $\xi \: Y \supset D_\xi \rightarrow X$ which satisfy the following conditions;
	\begin{enumerate}[(i)]
		\item $\xi\inv(D^\alpha_s) = D^\beta_{\theta(s)}$ for every $s \in S$,
		\item $\alpha_s(\xi(y)) = \xi(\beta_{\theta(s)}(y))$ for every $s \in S$ and $y\in \xi\inv(D_s^\alpha)$.
	\end{enumerate}
	We denote an \isamor $(\theta,\xi)$ from $(S,X,\alpha)$ to $(T,Y,\beta)$ as $(\theta,\xi)\: (S,X,\alpha) \rightarrow (T,Y,\beta)$.
\end{defi}

\begin{rema}\label{Remark: isamor}
	For every \isamor $(\theta,\xi)\: (S,X,\alpha) \rightarrow (T,Y,\beta)$, we get 
	\[
	\bigcup_{s \in S} D_{\theta(s)}^\beta = \bigcup_{s \in S} \xi\inv(D_s^\alpha) = \xi\inv\Big( \bigcup_{s \in S} D_s^\alpha \Big) = \xi\inv(X) = D_\xi
	\]
	by (i) in Definition \ref{Definition: isamor}.
	The range of $\beta_{\theta(s)}$ coincides with $D^\beta_{\theta(s^*)}$.
	The above equation implies $D^\beta_{\theta(s^*)} \subset D_\xi$.
	Hence 
	\[
	\xi\inv(D_s^\alpha) = D_{\theta(s)}^\beta = \beta_{\theta(s)}\inv(D_\xi)
	\]
	holds.
	Thus this equation and (ii) in Definition \ref{Definition: isamor} imply that 
	\[
	\alpha_s \circ \xi = \xi \circ \beta_{\theta(s)}
	\]
	as partial maps for every $s \in S$.
\end{rema}

\begin{rema}
	The condition that $\alpha_s \circ \xi = \xi \circ \beta_{\theta(s)}$ for every $s \in S$ does not implies (i) in Definition \ref{Definition: isamor}.
	Let both $S,T$ be the trivial groups and $X,Y$ be any topological spaces.
	We make the trivial groups act on $X$ and $Y$ by the identity maps on $X$ and $Y$ respectively. 
	Let $\theta\: S \rightarrow T$ be the identity semigroup homomorphism and $\xi\: Y \supset D_\xi \rightarrow X$ be any partial map with $D_\xi \subsetneq Y$.
	These satisfy that $\alpha_s \circ \xi = \xi \circ \beta_{\theta(s)}$ for every $s \in S$, but do not satisfy (i) in Definition \ref{Definition: isamor}.
\end{rema}

We illustrate an \isamor $(\theta,\xi) \: (S,X,\alpha) \rightarrow (T,Y,\beta)$ as below;
\[
\begin{tikzcd}
( S, \arrow[d,"\theta"'] &[-12mm] X, &[-12mm] \alpha )\\
( T, & Y, \arrow[u,"\xi"'] & \beta ).
\end{tikzcd}
\]

\begin{lemm}
	Let $(S_i,X_i,\alpha_i)$ be an inverse semigroup action for $i=1,2,3$ and $(\theta_i,\xi_i)\:(S_i,X_i,\alpha_i)\rightarrow (S_{i+1},X_{i+1},\alpha_{i+1})$ be an \isamor for $i=1,2$. 
	A pair $(\theta_2\circ \theta_1,\xi_1\circ \xi_2)$ becomes an \isamor from $(S_1,X_1,\alpha_1)$ to $(S_3,X_3,\alpha_3)$.
	\[
	\begin{tikzcd}
	( S_1, \arrow[d,"\theta_1"'] &[-12mm] X_1, &[-12mm] \alpha_1 )\\
	( S_2, \arrow[d,"\theta_2"'] & X_2, \arrow[u,"\xi_1"'] & \alpha_2 )\\
	( S_3, & X_3, \arrow[u,"\xi_2"'] & \alpha_3 ).
	\end{tikzcd}
	\]
\end{lemm}
\begin{proof}
	For every $s \in S_1$, we get
	\[
	(\xi_1\circ\xi_2)\inv(D^{\alpha_1}_s) 
	= \xi_2\inv(\xi_1\inv(D^{\alpha_1}_s)) 
	= \xi_2\inv\left(D^{\alpha_2}_{\theta_1(s)}\right) 
	= D^{\alpha_3}_{\theta_2(\theta_1(s))}.
	\]
	For every $s \in S_1$ and $x \in (\xi_1\circ\xi_2)\inv(D^{\alpha_1}_s) $, we get
	\[
	(\alpha_1)_{s}( \xi_1( \xi_2( x ) ) ) 
	= \xi_1( (\alpha_2)_{\theta_1(s)}( \xi_2( x ) ) ) 
	= \xi_1( \xi_2( (\alpha_3)_{\theta_2(\theta_1(s))}(x) ) ). \qedhere
	\]
\end{proof}

\begin{defi}\label{Definition: composition isamor}
	Let $(S_i,X_i,\alpha_i)$ be an inverse semigroup action for $i=1,2,3$ and $(\theta_i,\xi_i)\:(S_i,X_i,\alpha_i)\rightarrow (S_{i+1},X_{i+1},\alpha_{i+1})$ be an \isamor for $i=1,2$. 
	We define composition of the \isamors $(\theta_1,\xi_1)$ and $(\theta_2,\xi_2)$ as 
	\[
	(\theta_2,\xi_2) \cdot (\theta_1,\xi_1) := (\theta_2\circ \theta_1,\xi_1\circ \xi_2).
	\]
\end{defi}

\begin{theo}\label{Theorem: category ISA}
	All inverse semigroup actions $(S,X,\alpha)$ and all \isamors $(\theta,\xi)$ form a category with the composition in Definition \ref{Definition: composition isamor} and the identity morphisms $(\id_S,\id_X)$ for $(S,X,\alpha)$.
	We denote this category as $\ISA$.
\end{theo}
\begin{proof}
	It is easy to check that the associative law and the unit law.
\end{proof}

\begin{defi}
	An \isamor $(\theta, \xi)$ is said to be \emph{proper} if the partial map $\xi$ is proper.
\end{defi}

We consider inverse semigroup actions on locally compact Hausdorff spaces when constructing $C^*$-algebras from them.
It is natural to assume that \isamors between inverse semigroup actions on locally compact Hausdorff spaces are proper. 

\begin{theo}\label{Theorem: category ISAlcH}
	All inverse semigroup actions $(S,X,\alpha)$ on locally compact Hausdorff spaces $X$ and all proper \isamors $(\theta,\xi)$ form a subcategory of $\ISA$. 
	We denote this subcategory as $\ISAlcH$.
\end{theo}
\begin{proof}
	It is clear that identity \isamors are proper, and that compositions of proper \isamors are proper.
\end{proof}

\begin{exam}
	For a locally compact Hausdorff space $X$, we denote the set of all open subsets of $X$ as $\mathcal{O}(X)$.
	The set $\mathcal{O}(X)$ becomes an inverse semigroup with respect to intersection.
	The map $\id\: \mathcal{O}(X) \rightarrow I(X); U \mapsto \id_U$ is a semigroup homomorphism.
	For a proper partial map $\xi$ from a locally compact Hausdorff space $Y$ to $X$, the map $\xi\inv\:\mathcal{O}(X) \rightarrow \mathcal{O}(Y); U \mapsto \xi\inv(U)$ is a semigroup homomorphism.
	We can easily check that the pair $(\xi\inv,\xi)$ becomes a morphism from $(\mathcal{O}(X),X,\id)$ to $(\mathcal{O}(Y),Y,\id)$ in $\ISAlcH$.
	The constructions $X \mapsto (\mathcal{O}(X),X,\id)$ and $\xi \mapsto (\xi\inv,\xi)$ form a functor from $\LCHS\op$ to $\ISAlcH$.
	We get that this functor is full as follows:
	Let $X,Y$ be locally compact Hausdorff spaces, and $(\theta,\xi)$ be a morphism from $(\mathcal{O}(X),X,\id)$ to $(\mathcal{O}(Y),Y,\id)$ in $\ISAlcH$.
	By the condition (i) in Definition \ref{Definition: isamor}, we get $\theta(U) = \xi\inv(U)$ for every $U \in \mathcal{O}(X)$.
	Thus two semigroup homomorphisms $\theta$ and $\xi\inv$ from $\mathcal{O}(X)$ to $\mathcal{O}(Y)$ coincide.
\end{exam}

We construct a functor from $\IS$ to $\ISAlcH$ which sends inverse semigroups to their spectral actions.
Let $S$ and $T$ be inverse semigroups and $\theta\: S \rightarrow T$ be a semigroup homomorphism.
We define a continuous map $\theta_*$ from $\widehat{E(T)}_0$ to $\widehat{E(S)}_0$ as 
\[
\theta_*\: \widehat{E(T)}_0 \rightarrow \widehat{E(S)}_0; \zeta \mapsto \zeta\circ\theta.
\]
This map $\theta_*$ sends $0$ to $0$.

\begin{defi}
	We denote $\widehat{\theta}$ as the partial map from $\widehat{E(T)}$ to $\widehat{E(S)}$ induced from the continuous map $\theta_*$ from $\widehat{E(T)}_0$ to $\widehat{E(S)}_0$, that is,  $\widehat{\theta}$ is defined on the open subspace $D_{\widehat{\theta}} := \theta_*\inv\big(\widehat{E(S)}\big)$ of $\widehat{E(T)}$ as $\widehat{\theta}(\zeta) := \theta_*(\zeta)$ for $\zeta \in D_{\widehat{\theta}}$.
\end{defi}

\begin{prop}\label{Proposition: proper morphism}
	 The pair $\big(\theta,\widehat{\theta}\,\big)$ becomes a proper \isamor from $\big( S,\widehat{E(S)},\beta^S \big)$ to $\big( T,\widehat{E(T)},\beta^T \big)$.
\end{prop}
\begin{proof}
	We see that $\widehat{\theta}$ is proper by its construction. 
	For every $s \in S$ and $\zeta\in \widehat{E(T)}$,
	\[
	\theta_*\zeta(s^*s) = \zeta(\theta(s^*s)) = \zeta(\theta(s)^*\theta(s))
	\]
	holds. 
	Thus we get $\widehat{\theta}\inv\big(U_s^S\big) = U^T_{\theta(s)}$.
	For every $s \in S$, $\zeta \in \widehat{\theta}\inv\big(U_s^S\big)$, and $e \in E(S)$,
	\begin{align*}
	\beta_s^S\big(\widehat{\theta}(\zeta)\big)(e)
	= \widehat{\theta}\zeta(s^*es)
	&= \zeta(\theta(s^*es))\\
	&= \zeta(\theta(s)^*\theta(e)\theta(s))
	= \big(\beta_{\theta(s)}^T\zeta\big)(\theta(e))
	= \widehat{\theta}\big(\beta_{\theta(s)}^T\zeta\big)(e)
	\end{align*}
	holds. 
	Thus we get $\beta_s^S\big(\widehat{\theta}(\zeta)\big) = \widehat{\theta}\big(\beta_{\theta(s)}^T\zeta\big)$ for every $s \in S$ and $\zeta \in \widehat{\theta}\inv\big(U_s^S\big)$.
\end{proof}

\begin{theo}\label{Theorem: functor SP}
	The constructions $S \mapsto \big(S,\widehat{E(S)},\beta^S\big)$ and $\theta \mapsto \big(\theta,\widehat{\theta}\,\big)$ form a functor from $\IS$ to $\ISAlcH$.
	We denote this functor as $\SP$.
\end{theo}
\begin{proof}
	For an inverse semigroup $S$, the space $\widehat{E(S)}$ is locally compact Hausdorff.  
	It is easy to check that the identity homomorphism for $S$ induces the identity \isamor for $\big( S, \widehat{E(S)}, \beta^S \big)$.
	By Proposition \ref{Proposition: proper morphism}, $\big(\theta,\widehat{\theta}\,\big)$ is a proper \isamor for a semigroup homomorphism $\theta$. 
	
	Let $\theta_i\: S_i \rightarrow S_{i+1}$ be a semigroup homomorphism for $i = 1,2$.
	The maps $(\theta_1)_* \circ (\theta_2)_*$ and $(\theta_2\circ\theta_1)_*$ are the same continuous map from $\widehat{E(S_3)}_0$ to $\widehat{E(S_1)}_0$.
	This implies that $\widehat{\theta_1}\circ\widehat{\theta_2} = \widehat{\theta_2\circ\theta_1}$.
	Hence we get $\big(\theta_2,\widehat{\theta_2}\big)\cdot\big(\theta_1,\widehat{\theta_1}\big) = \big(\theta_2\circ\theta_1,\widehat{\theta_2\circ\theta_1}\big)$.
\end{proof}

\section{A category $\EG$}
\label{Section: A category EG}

\subsection{Two classes of groupoid homomorphisms}

In this subsection, we investigate two properties for groupoid homomorphisms.
We first investigate groupoid homomorphisms whose restrictions to each fibers are bijective:

\begin{defi}
	Let $G$ and $K$ be groupoids.
	A groupoid homomorphism $\varphi \: K \rightarrow G$ is said to be \emph{fibrewise bijective} if the restriction $\varphi_u \: K_u\rightarrow G_{\varphi(u)}$ of $\varphi$ is bijective for all $u\in K\0$.
\end{defi}

The notion ``fibrewise bijective'' is borrowed from \cite[Definition~6.1]{Tay23}. 
In the above definition, we can easily see that $\varphi$ is fibrewise bijective if and only if the restriction $\varphi^u\: K^U \rightarrow G^{\varphi(u)}$ of $\varphi$ is bijective for all $u\in K\0$.

\begin{rema}\label{Remark: composition fib bij}
	Let $G_i$ be a groupoid with $i = 1,2,3$, and $\varphi_i\: G_i\rightarrow G_{i+1}$ be a fibrewise bijective groupoid homomorphism with $i=1,2$.
	For every $u\in G_1\0$, $(\varphi_2\circ\varphi_1)_u = (\varphi_2)_{\varphi_1(u)}\circ(\varphi_1)_u$ is bijective. 
	Thus $\varphi_2\circ\varphi_1$ is fibrewise bijective.
\end{rema}

\begin{lemm}\label{Lemma: varphi inv}
	Let $G,K$ be groupoids, and $\varphi\: K \rightarrow G$ be a fibrewise bijective groupoid homomorphism. 
	For every subsets $U$ and $V$ of $G$, the following hold:
	\begin{enumerate}[(i)]
		\item $\varphi\inv(UV) = \varphi\inv(U)\varphi\inv(V)$.
		\item $\varphi\inv \left(G\0\right) = K\0$.
		\item $\varphi\inv(d(U)) = d\left(\varphi\inv(U)\right)$.
		\item If $U$ is a bisection of $G$, $\varphi\inv(U)$ is a bisection of $K$.
	\end{enumerate}
\end{lemm}
\begin{proof}
	For every groupoid homomorphism $\varphi\:K\rightarrow G$, we can check easily that the left-hand sides in (i)-(iii) include the right-hand sides. 
	We prove that the opposite inclusions hold if $\varphi_u$ is bijective for all $u\in K\0$. 
	\begin{enumerate}[(i)]
		\item Take an element $k\in \varphi\inv\left(UV\right)$. 
		There exists a composable pair $g_1\in U$ and $g_2\in V$ such that $\varphi(k) = g_1g_2$. 
		We set $d(k) = u$. 
		Since $g_2 \in G_{\varphi(u)}$ and $\varphi_u$ is surjective, there exists an element $k_2\in K_u$ such that $\varphi\left(k_2\right) = g_2$. 
		Setting $k_1 := kk_2\inv$, we get $\varphi(k_1) = g_1$. Thus $k = k_1k_2\in \varphi\inv(U)\varphi\inv(V)$ holds.
		
		\item Take $k\in \varphi\inv\left(G\0\right)$ and set $u:=d(k)$. 
		Since $\varphi_u(d(k)) = d(\varphi_u(k)) = \varphi_u(k)$ holds and $\varphi_u$ is injective, $k = d(k) \in K\0$ holds. 
		
		\item 
		Take $u\in \varphi\inv(d(U))$. Take an element $g\in U$ satisfying $\varphi(u) = d(g)$.
		By the equation (ii), we get $u\in K\0$. 
		Since $g\in G_{\varphi(u)}$ and the restriction $\varphi_u\:K_u\rightarrow G_{\varphi(u)}$ of $\varphi$ is surjective, there exists $k\in K_u$ satisfying $\varphi_u(k) = g$. 
		Thus $u = d(k)\in d\big(\varphi\inv(U)\big)$ holds.
		
	\end{enumerate}	
	
	We check the statement (iv).
	We assume that $U$ is a bisection of $G$.
	Taking elements $k,k'\in \varphi\inv(U)$ with $d(k) = d(k')$, we get $d(\varphi(k)) = d(\varphi(k'))$. 
	Since $U$ is a bisection, we get $\varphi(k) = \varphi(k')$.
	It implies that $k=k'$ because $\varphi_u$ is injective where $u = d(k) = d(k')$.
	Similarly we can see that $r(k) = r(k')$ implies $k = k'$ for $k,k'\in\varphi\inv(U)$ because $\varphi^u\: K^u\rightarrow G^{\varphi(u)}$ is injective for $u := r(k) = r(k')$.
\end{proof}


Let $G, K$ be \'etale groupoids, and $\varphi\: K \rightarrow G$ be a continuous fibrewise bijective groupoid homomorphism.
\begin{rema}\label{Remark: phi homeo iff phi0 homeo}
	By Lemma \ref{Lemma: grpd hom inj surj} (v) and \ref{Lemma: phi open iff phi0 open}, $\varphi$ is homeomorphic if and only if so is $\varphi\0$.
\end{rema}

By Lemma \ref{Lemma: varphi inv} (i) and (iv) , the map which sends an open bisection $U$ of $G$ to an open bisection $\varphi\inv(U)$ of $K$ is regarded as a semigroup homomorphism from $\Bis G$ to $\Bis K$. 
\begin{defi}\label{Definition: varphi inv}
	We denote the semigroup homomorphism 
	\[
	\Bis G \rightarrow \Bis K ; U \mapsto \varphi\inv(U)
	\]
	as $\varphi\inv$. 
\end{defi}

Constructing a \shom from $\varphi\: K \rightarrow G$ in Section \ref{Section: A functor C* from EG to Calg} in mind, we investigate the properness for $\varphi\0$.


\begin{lemm}\label{Lemma: phi proper}
	If $\varphi\0$ is proper, then the map $\varphi\: \varphi\inv(U) \rightarrow U$ is proper for every open bisection $U$ of $G$.
\end{lemm}
\begin{proof}
	Take a compact subset $C$ of $U$.	
	By Lemma \ref{Lemma: dr homeo on op bis} and Lemma \ref{Lemma: varphi inv} (iv), $C$ and $\varphi\inv(C)$ are homeomorphic to $d(C)$ and $d(\varphi\inv(C))$ respectively. 
	By Lemma \ref{Lemma: varphi inv} (iii), $d(\varphi\inv(C)) = \varphi\inv(d(C)) \subset K\0$ holds.
	This subset is compact since $\varphi\0$ is proper.
	Thus $\varphi\inv(C)$ is compact.
\end{proof}

Secondly, we investigate groupoid homomorphisms which are injective on unit spaces.



\begin{lemm}\label{Lemma: psi}
	Let $K,H$ be groupoids, and $\psi\: K \rightarrow H$ be a groupoid homomorphism with injective $\psi\0$.
	For subsets $U$ and $V$ of $K$, $\psi(UV) = \psi(U)\psi(V)$ holds.
	For a bisection $U$ of $K$, $\psi(U)$ is a bisection of $H$, and the map $\psi\: U \rightarrow \psi(U)$ is injective.
\end{lemm}
\begin{proof}
	Since $\psi$ is a groupoid homomorphism, $\psi(UV) \subset \psi(U)\psi(V)$ holds. 
	Take elements $k\in U$ and $k'\in V$ such that $d(\psi(k)) = r(\psi(k'))$. 
	Since $\psi\0$ is injective, $d(k) = r(k')$ holds. 
	Thus $\psi(k)\psi(k') = \psi(kk') \in \psi(UV)$ holds.
	Hence we get $\psi(UV) = \psi(U)\psi(V)$.
	
	We take a bisection $U$ of $K$.
	The maps $d\:U \rightarrow K\0$ and $\psi\0\: K\0 \rightarrow H\0$ are injective.
	The composition $\psi\0 \circ d \: U \rightarrow H\0$ is injective.
	This map coincides with the restriction of $d\circ\psi$ to $U$.
	Thus we get $\psi\: U \rightarrow \psi(U)$ and $d\: \psi(U) \rightarrow H\0$ are injective.
	In the same way, we get the restriction of $r$ to $\psi(U)$ is injective. 
	Hence $\psi(U)$ is a bisection. 
\end{proof}

Let $K,H$ be \'etale groupoids, and $\psi\: K\rightarrow H$ be a continuous groupoid homomorphism with open injective $\psi\0$.


\begin{lemm}\label{Lemma: psi 2}
	For an open bisection $U$ of $K$, $\psi(U)$ is an open bisection of $H$, and
	the map $\psi|_{U}\: U \rightarrow \psi(U)$ is a homeomorphism.
\end{lemm}
\begin{proof}
	This follows from Lemma \ref{Lemma: phi open iff phi0 open} and \ref{Lemma: psi}.
\end{proof}

By Lemma \ref{Lemma: psi} and \ref{Lemma: psi 2}, the map which sends an open bisection $U$ of $K$ to an open bisection $\psi(U)$ of $H$ can be regarded as a semigroup homomorphism from $\Bis K$ to $\Bis H$. 
\begin{defi}\label{Definition: psi}
	We denote the semigroup homomorphism 
	\[
	\Bis K \rightarrow \Bis H; U\mapsto\psi(U)
	\]
	as $\psi$.
\end{defi}

\begin{rema}\label{Remark: comp psi0 open inj}
	Let $G_i$ be an \'etale groupoid with $i = 1,2,3$, and $\psi_i\:G_i \rightarrow G_{i+1}$ be a continuous groupoid homomorphism with open injective $\psi_i\0$ with $i = 1,2$.
	It is clear that $\psi_2\circ\psi_1$ has open injective $(\psi_2\circ\psi_1)\0$ because $(\psi_2\circ\psi_1)\0 = \psi_2\0 \circ \psi_1\0$.
\end{rema}

\subsection{A category $\EG$ of \'etale groupoids}
\label{Subsection: A category EG of etale groupoids}

We prepare lemmas and propositions on groupoid homomorphisms to defining morphism between groupoids.
Let $G,H$ be groupoids.

\begin{lemm}\label{Lemma: egmors subset}
	Let $K$ be a groupoid, $\varphi\: K \rightarrow G$ be a groupoid homomorphism such that $\varphi_u\: K_u \rightarrow G_{\varphi(u)}$ is injective for all $u \in K\0$, and $\psi\: K \rightarrow H$ be a groupoid homomorphism with injective $\psi\0$.
	\begin{enumerate}[(i)]
		\item For a subset $V$ of $G\0$ and an element $u$ of $K\0$, $\varphi(u)\in V$ holds if and only if $\psi(u) \in \psi(\varphi\inv(V))$ holds.
		\item The groupoid homomorphism $\varphi\times\psi$ defined as
		\[
		\varphi\times\psi \: K \rightarrow G\times H; k \mapsto (\varphi(k),\psi(k))
		\]
		is injective.
	\end{enumerate}
\end{lemm}
\begin{proof}
	\begin{enumerate}[(i)]
		\item The only if part is clear. 
		We show that the if part. 
		Take $u\in K\0$ and $V \subset G\0$ such that $\psi(u) \in \psi(\varphi\inv(V))$ holds.
		There exists $u'\in \varphi\inv(V)$ such that $\psi(u) = \psi(u')$. 
		By Lemma \ref{Lemma: varphi inv} (ii), $\varphi\inv(V)$ is included in $K\0$ (Lemma \ref{Lemma: varphi inv} (ii) holds only by assuming that $\varphi_u$ is injective for every $u \in K\0$). 
		Hence we get $u'\in K\0$.
		Since $\psi\0$ is injective, $\psi(u) = \psi(u')$ implies $u=u'$.
		Thus $\varphi(u) = \varphi(u') \in V$ holds.
		\item Take $k,k'\in K$ with $\varphi(k) = \varphi(k')$ and $\psi(k) = \psi(k')$.
		Since $\psi\0$ is injective and $\psi\0(d(k)) = \psi\0(d(k'))$, $d(k) = d(k')$ holds.
		We set $u\in K\0$ as the element $d(k)(= d(k'))$.
		Since $\varphi_u$ is injective and $\varphi_u(k) = \varphi_u(k')$, we get $k = k'$.
		Thus $\varphi\times \psi$ is injective. \qedhere
	\end{enumerate}
\end{proof}

\begin{lemm}\label{Lemma: auto grpd hom}
	Let $K_1,K_2,K$ be groupoids, and $\varphi_1\: K_1 \rightarrow K$, $\varphi_2\: K_2 \rightarrow K$ be groupoid homomorphisms.
	A map $\iota \: K_1 \rightarrow K_2$ with $\varphi_2 \circ \iota = \varphi_1$ becomes a groupoid homomorphism if $\varphi_2$ is injective.
\end{lemm}
\begin{proof}
	Take $k,k'\in K_1$ with $d(k) = r(k')$.
	We get  
	\begin{align*}
	\varphi_2\0(d(\iota(k))) &= d(\varphi_2(\iota(k))) = d(\varphi_1(k)) \\
	&= r(\varphi_1(k')) = r(\varphi_2(\iota(k'))) = \varphi_2\0(r(\iota(k'))).
	\end{align*}
	Since $\varphi_2\0$ is injective, we have $d(\iota(k)) = r(\iota(k'))$.
	We get 
	\begin{align*}
	\varphi_2(\iota(kk')) &= \varphi_1(kk') = \varphi_1(k)\varphi_1(k') \\
	&= \varphi_2(\iota(k))\varphi_2(\iota(k')) = \varphi_2(\iota(k)\iota(k')).
	\end{align*}
	Since $\varphi_2$ is injective, we have $\iota(kk') = \iota(k)\iota(k')$.
\end{proof}

\begin{lemm}\label{Lemma: map btwn alg cpl mor} 
	Let $K_i$ be a groupoid, and $\varphi_i\: K_i\rightarrow G$, $\psi_i\: K_i\rightarrow H$ be groupoid homomorphisms with $i=1,2$. 
	For a map $\iota\: K_1\rightarrow K_2$ which makes the diagram
	\[
	\begin{tikzcd}
	&K_1\arrow[ld,"\varphi_1"']\arrow[rd,"\psi_1"]\arrow[dd,"\iota"]&\\[-3mm]
	G&&H\\[-3mm]
	&K_2\arrow[lu,"\varphi_2"]\arrow[ru,"\psi_2"']&
	\end{tikzcd}
	\]
	commute, the following hold:
	\begin{enumerate}[(i)] 
		\item If $(\varphi_2)_v$ is injective for every $v \in K_2\0$ and $\psi_2\0$ is injective, then $\iota$ is a unique map which makes the above diagram commute and becomes automatically a groupoid homomorphism.
		\item If $(\varphi_1)_u$ is injective for every $u \in K_1\0$ and $\psi_1\0$ is injective in addition to (i), then $\iota$ is injective. 
		\item If $(\varphi_1)_u$ is surjective for every $u \in K_1\0$ in addition to (ii), then $\iota_u$ is bijective for every $u \in K_1\0$.
		\item If $\iota\big(K_1\0\big) = K_2\0$ in addition to (iii), then $\iota$ is a unique bijective groupoid homomorphism which makes the above diagram commute.
	\end{enumerate}
\end{lemm}
\begin{proof}
	We get that the map $\iota$ satisfies $(\varphi_2\times\psi_2) \circ \iota = (\varphi_1\times\psi_1)$.
	\begin{enumerate}[(i)]
		\item By Lemma \ref{Lemma: egmors subset} (ii), $\varphi_2\times\psi_2$ is injective.
		Thus $\iota$ is unique.
		By Lemma \ref{Lemma: auto grpd hom}, $\iota$ becomes a groupoid homomorphism.
		\item By Lemma \ref{Lemma: egmors subset} (ii), $\varphi_1\times\psi_1$ is injective.
		Thus $\iota$ is injective.
		\item We get that $\iota_u$ is bijective by $(\varphi_1)_u = (\varphi_2)_{\iota(u)} \circ \iota_u$ for every $u \in K_1\0$.
		\item The map $\iota$ is bijective by Lemma \ref{Lemma: grpd hom inj surj} (v). \qedhere
	\end{enumerate}
\end{proof}

\begin{prop}\label{Proposition: egmors alg}
	Let $K_i$ be a groupoid, $\varphi_i\: K_i \rightarrow G$ be a fibrewise bijective groupoid homomorphism, and $\psi_i\: K_i \rightarrow H$ be a groupoid homomorphism with injective $\psi_i\0$ for $i = 1,2$.
	The following are equivalent:
	\begin{enumerate}[(i)]		
		\item There exists a map $\iota\: K_1\rightarrow K_2$ which makes the following diagram commute:
		\[
		\begin{tikzcd}
			&K_1\arrow[ld,"\varphi_1"']\arrow[rd,"\psi_1"]\arrow[dd,"\iota"]&\\[-3mm]
			G&&H.\\[-3mm]
			&K_2\arrow[lu,"\varphi_2"]\arrow[ru,"\psi_2"']&
		\end{tikzcd}
		\]
		\item The image of $\varphi_1\times \psi_1$ is included in the one of $\varphi_2\times \psi_2$ in $G \times H$.
		\item For every subset $U$ of $G$, $\psi_1\big(\varphi_1\inv(U)\big) \subset \psi_2\big(\varphi_2\inv(U)\big)$ holds.		
	\end{enumerate}
	Moreover, a map $\iota$ in (i) is unique and becomes automatically an injective groupoid homomorphism.
\end{prop}
\begin{proof}
	It is easy to check that (ii) $\Leftrightarrow$ (i) $\Rightarrow$ (iii).
	We show (iii) $\Rightarrow$ (ii)
	We fix $k_1\in K_1$ arbitrarily.
	We set $U := \{\varphi_1(k_1)\}\subset G$.
	By assumption, we get $\psi_1(k_1) \in \psi_1\big(\varphi_1\inv(U)\big) \subset \psi_2\big(\varphi_2\inv(U)\big)$.
	Thus there exists $k_2\in K_2$ such that $\psi_1(k_1) = \psi_2(k_2)$ and $\varphi_2(k_2)\in U=\{\varphi_1(k_1)\}$.
	We get 
	\[
	(\varphi_1\times\psi_1)(k_1) = (\varphi_1(k_1),\psi_1(k_1)) = (\varphi_2(k_2),\psi_2(k_2)) \in (\varphi_2\times\psi_2)(K_2).
	\]
	Hence (ii) holds.
	
	The last statement follows from Lemma \ref{Lemma: map btwn alg cpl mor} (i), (ii).
\end{proof}

\begin{coro}\label{Corollary: egmor equiv alg}
	Let $K_i$ be a groupoid, $\varphi_i\: K_i \rightarrow G$ be a fibrewise bijective groupoid homomorphism, and $\psi_i\: K_i \rightarrow H$ be a groupoid homomorphism with injective $\psi_i\0$ for $i = 1,2$.
	The following are equivalent:
	\begin{enumerate}[(i)]
		\item There exists a map $\iota \: K_1 \rightarrow K_2$ which makes the diagram in Proposition \ref{Proposition: egmors alg} (i) commute and satisfies $\iota\big(K_1\0\big)=K_2\0$.
		\item The image of $\varphi_1\times \psi_1$ coincides with the one of $\varphi_2\times \psi_2$ in $G \times H$.
		\item For every subset $U$ of $G$, $\psi_1\big(\varphi_1\inv(U)\big) = \psi_2\big(\varphi_2\inv(U)\big)$ holds.
	\end{enumerate}
\end{coro}
\begin{proof}
	By Proposition \ref{Proposition: egmors alg}, both (ii) and (iii) are equivalent to the following;
	\begin{enumerate}
		\item[(i')] there exist a map $\iota \: K_1 \rightarrow K_2$ such that $\varphi_2\circ\iota = \varphi_1$ and $\psi_2\circ\iota = \psi_1$, and a map $\iota' \: K_2 \rightarrow K_1$ such that $\varphi_1\circ\iota' = \varphi_2$ and $\psi_1\circ\iota' = \psi_2$.
	\end{enumerate}
	The proof finishes if we show that (i) is equivalent to (i'). 
	Since Lemma \ref{Lemma: map btwn alg cpl mor} (iv) implies that the map $\iota$ in (i) is bijective,
	(i) implies (i').
	Assume that (i') holds. 
	Then $\iota$ and $\iota'$ are the inverses of each other by the uniqueness of $\iota$ in Proposition \ref{Proposition: egmors alg} (i).
	Thus $\iota$ is bijective.
	This implies (i).
	\end{proof}

Let $G,H$ be \'etale groupoids.

\begin{defi}
	A \emph{\egmor $(\varphi,\psi;K)$ from $G$ to $H$} is a couple of a continuous fibrewise bijective groupoid homomorphism $\varphi \: K \rightarrow G$ and a continuous groupoid homomorphism $\psi \: K \rightarrow H$ with open injective $\psi\0$, where $K$ is an \'etale groupoid.
	We write $(\varphi,\psi;K)$ from $G$ to $H$ as $(\varphi,\psi;K)\: G \rightarrow H$.
\end{defi}

\begin{defi}
	Let $(\varphi_i,\psi_i;K_i)$ be a \egmor from $G$ to $H$ with $i=1,2$.
	An \emph{equivalence} from $(\varphi_1,\psi_1;K_1)$ to $(\varphi_2,\psi_2;K_2)$ is a homeomorphic groupoid isomorphism $\iota$ from $K_1$ to $K_2$ which makes the diagram
	\[
	\begin{tikzcd}
		&K_1\arrow[ld,"\varphi_1"']\arrow[rd,"\psi_1"]\arrow[dd,"\iota"]&\\[-3mm]
		G&&H.\\[-3mm]
		&K_2\arrow[lu,"\varphi_2"]\arrow[ru,"\psi_2"']&
	\end{tikzcd}
	\]
	commute.
	Two \egmors $(\varphi_1,\psi_1;K_1)$ and $(\varphi_2,\psi_2;K_2)$ are said to be \emph{equivalent} if there exists an equivalence between them.
\end{defi}

We will regard the equivalence classes of \egmors as morphisms of the category $\EG$ of \'etale groupoids in Theorem \ref{Theorem: category EG}.
We prepare lemmas and propositions to give necessary and sufficient conditions in Corollary \ref{Corollary: egmor equiv} so that two \egmors are equivalent.

\begin{lemm}\label{Lemma: phipsi homeo}
	Let $(\varphi,\psi;K)$ be a \egmor from $G$ to $H$. 
	The groupoid homomorphism $\varphi\times\psi \: K \rightarrow G \times H$ is a homeomorphic groupoid isomorphism onto its image.
\end{lemm}
\begin{proof}
	By Lemma \ref{Lemma: egmors subset} (ii), $\varphi\times\psi$ is injective.
	We can easily check that $\varphi\times\psi$ is continuous.
	The proof finishes if we show that $(\varphi\times\psi)(U)$ is open in $(\varphi\times\psi)(K)$ for an open subset $U$ of $K$. 
	To this end, it suffices to find for each element $k_0 \in U$ open subsets $V$ and $W$ of $G$ and $H$ respectively such that 
	$(\varphi\times\psi)(k_0) \in (V \times W) \cap (\varphi\times\psi)(K) \subset (\varphi\times\psi)(U)$. 
	Take an element $k_0 \in U$. 
	By Lemma \ref{Lemma: bisection}, there exists an open bisection $V$ of $G$ containing $\varphi(k_0)$. 
	The subset $U \cap \varphi^{-1}(V)$ is an open subset of $K$ containing $k_0$. 
	Hence the subset $W := \psi(U \cap \varphi^{-1}(V))$ is an open subset of $W$ containing $\psi(k_0)$ by Lemma \ref{Lemma: phi open iff phi0 open}. 
	We have $(\varphi\times\psi)(k_0) \in (V \times W) \cap (\varphi\times\psi)(K)$. 
	We show $(V \times W) \cap (\varphi\times\psi)(K) \subset (\varphi\times\psi)(U)$. 
	An element in $(V \times W) \cap (\varphi\times\psi)(K)$ is in the form $(\varphi(k),\psi(k))$ for $k \in K$ with $\varphi(k)\in V$ and $\psi(k) \in W$. 
	Since $\psi(k) \in W$, there exists $k' \in U \cap \varphi^{-1}(V)$ with $\psi(k) = \psi(k')$. 
	Since $\psi\0$ is injective and $\psi\0(d(k)) = \psi\0(d(k'))$, $d(k) = d(k')$ holds.
	Hence we have $d(\varphi(k))=d(\varphi(k'))$. 
	Since both $\varphi(k)$ and $\varphi(k')$ are in the bisection $V$, we get $\varphi(k) = \varphi(k')$. 
	This implies $k = k'$ because $\varphi_u$ is injective for $u := d(k)(= d(k'))$.
	Thus we get $k \in U$. 
	Therefore we have $(V \times W) \cap (\varphi\times\psi)(K) \subset (\varphi\times\psi)(U)$. 
\end{proof}

Let $(\varphi_i,\psi_i;K_i)$ be a \egmor from $G$ to $H$ for $i = 1,2$. 

\begin{prop}\label{Proposition: egmors}
	The conditions (i)-(iii) in Proposition \ref{Proposition: egmors alg} are equivalent.
	The map $\iota$ in (i) is unique, and becomes automatically a homeomorphic groupoid isomorphism onto its image.
\end{prop}
\begin{proof}
	By Lemma \ref{Lemma: map btwn alg cpl mor}, the map $\iota$ is a unique map which makes the diagram commute and becomes a groupoid homomorphism.
	By Lemma \ref{Lemma: phipsi homeo}, $\varphi_1\times\psi_1 \: K_1 \rightarrow (\varphi_1\times\psi_1)(K_1)$ and $(\varphi_2\times\psi_2)\inv \: (\varphi_2\times\psi_2)(K_2) \rightarrow K_2$ are homeomorphic groupoid isomorphisms.
	By the uniqueness, $\iota$ coincides with the homeomorphic groupoid isomorphism $(\varphi_2\times\psi_2)\inv \circ (\varphi_1\times\psi_1)$ onto its image.
\end{proof}

\begin{coro}\label{Corollary: egmor equiv}
	\Egmors $(\varphi_1,\psi_1;K_1)$ and $(\varphi_2,\psi_2;K_2)$ are equivalent if and only if one of the three equivalent conditions (i)-(iii) in Corollary \ref{Corollary: egmor equiv alg} holds.
\end{coro}

\begin{prop}
	If $G\0$ is Hausdorff, the conditions (i)-(iii) in Proposition \ref{Proposition: egmors alg} are equivalent to the following;
	\begin{enumerate}
		\item[(iv)]  For every open bisection $U$ of $G$, $\psi_1\big(\varphi_1\inv(U)\big) \subset \psi_2\big(\varphi_2\inv(U)\big)$ holds.		
	\end{enumerate}
\end{prop}
\begin{proof}
	We get (iii) $\Rightarrow$ (iv) clearly.
	We show (iv) $\Rightarrow$ (ii).
	We fix $k_1\in K_1$ arbitrarily.
	There exists an open bisection $V$ of $G$ with $\varphi_1(k_1)\in V$ by Lemma \ref{Lemma: bisection}.
	By assumption, we get $\psi_1(k_1) \in \psi_1\big(\varphi_1\inv(V)\big) \subset \psi_2\big(\varphi_2\inv(V)\big)$.
	Thus there exists $k_2\in K_2$ such that $\psi_1(k_1) = \psi_2(k_2)$ and $\varphi_2(k_2)\in V$.
	
	If $\varphi_1(d(k_1)) \in U$ holds for an open subset $U$ of $G\0$, then $\psi_1(d(k_1)) \in \psi_1(\varphi_1\inv(U))$ holds.
	By assumption, we get
	\[
	\psi_2(d(k_2)) = \psi_1(d(k_1)) \in \psi_1(\varphi_1\inv(U)) \subset \psi_2\big(\varphi_2\inv(U)\big).
	\]
	This implies $\varphi_2(d(k_2)) \in U$ by Lemma \ref{Lemma: egmors subset} (i).
	
	By the above discussion, we have the fact that $\varphi_1(d(k_1)) \in U$ implies $\varphi_2(d(k_2)) \in U$ for every open subset $U$ of $G\0$.
	Since $G\0$ is Hausdorff, we get $\varphi_1(d(k_1)) = \varphi_2(d(k_2))$.
	Hence we get $\varphi_1(k_1) = \varphi_2(k_2)$ since both $\varphi_1(k_1)$ and $\varphi_2(k_2)$ are included in the same bisection $V$.
	Thus we have
	\[
	(\varphi_1\times\psi_1)(k_1) = (\varphi_1(k_1),\psi_1(k_1)) = (\varphi_2(k_2),\psi_2(k_2)) \in (\varphi_2\times\psi_2)(K_2).
	\]
	Hence (ii) holds.
\end{proof}

\begin{coro}
	If $G\0$ is Hausdorff, the conditions (i)-(iii) in Corollary \ref{Corollary: egmor equiv alg} are equivalent to the following;
	\begin{enumerate}
		\item[(iv)]  For every open bisection $U$ of $G$, $\psi_1\big(\varphi_1\inv(U)\big) = \psi_2\big(\varphi_2\inv(U)\big)$ holds.		
	\end{enumerate}
\end{coro}

\begin{defi}
	We denote the equivalence class of a \egmor $(\varphi,\psi;K)$ from $G$ to $H$ as $[\varphi,\psi;K] \: G \rightarrow H$, or just $[\varphi,\psi;K]$, or $[\varphi,\psi]$ for simplicity. 
	We call an equivalence class $[\varphi,\psi;K]$ also as a \egmor.
\end{defi}

To define composition of (equivalence classes of) \egmors, we recall the pull-back of groupoid homomorphisms.
Let $G$, $K_1$, $K_2$ be groupoids, and $\psi\: K_1\rightarrow G$, $\varphi\:K_2\rightarrow G$ be groupoid homomorphisms.
The \emph{pull-back groupoid} is constructed as follows: 
We define the arrow space as
\[
K_1\times_{G} K_2:=\{ (k_1,k_2)\in K_1\times K_2 \mid \psi(k_1)=\varphi(k_2) \},
\]
and the unit space as 
\[
\big(K_1\times_{G} K_2\big)\0 := \big(K_1\times_{G} K_2\big) \cap \big(K_1\0 \times K_2\0\big).
\]
The domain map, the range map, the multiplication, and the inverse map are defined component-wise. 
We define groupoid homomorphisms $\tilde{\varphi}\: K_1\times_{G} K_2\rightarrow K_1$ and $\tilde{\psi}\: K_1\times_{G} K_2\rightarrow K_2$ as
\[
\tilde{\varphi}(k_1,k_2) := k_1, \hspace{5mm} \tilde{\psi}(k_1,k_2) := k_2, 
\]
for $(k_1,k_2) \in K_1\times_{G} K_2$. 
The diagram 
\[
\begin{tikzcd}
&K_1\times_{G} K_2 \arrow[dl,"\widetilde{\varphi}"']\arrow[dr,"\widetilde{\psi}"]&\\[-5mm]
K_1\arrow[dr,"\psi"']&&K_2\arrow[dl,"\varphi"]\\[-5mm]
&G&
\end{tikzcd}
\]
is commutative.


\begin{lemm}\label{Lemma: pull-back 1}
	For a subset $U \subset K_1$, we have 
	\[
	\varphi\inv\big( \psi(U) \big) = \tilde{\psi}\big( \tilde{\varphi}\inv(U) \big).
	\]
\end{lemm}
\begin{proof}
	Take $k_2 \in \widetilde{\psi}\big(\widetilde{\varphi}\inv(U)\big)$ arbitrarily.
	There exists $k \in K_1\times_G K_2$ such that $k_2 = \widetilde{\psi}(k)$ and $\widetilde{\varphi}(k) \in U$.
	We get $\varphi(k_2) = \varphi\big(\widetilde{\psi}(k)\big) = \psi(\widetilde{\varphi}(k)) \in \psi(U)$.
	Hence we have $k_2 \in \varphi\inv(\psi(U))$.
	
	Take $k_2 \in \varphi\inv(\psi(U))$ arbitrarily.
	There exists $k_1\in U$ with $\psi(k_1) = \varphi(k_2)$.
	We get $(k_1,k_2) \in K_1\times_G K_2$ and $\widetilde{\varphi}(k_1,k_2) = k_1 \in U$.
	Hence we have $k_2 = \widetilde{\psi}(k_1,k_2) \in \widetilde{\psi}\big(\widetilde{\varphi}\inv(U)\big)$.
\end{proof}

\begin{lemm}\label{Lemma: pull-back fibrewise}
	If $\varphi$ is fibrewise bijective, then so is $\widetilde{\varphi}$.
\end{lemm}
\begin{proof}
	The map $\tilde{\varphi}_{(u,v)}\: \left(K_1\times_G K_2\right)_{(u,v)} \rightarrow (K_1)_u$ is a bijection because a map defined as 
	\begin{align*}
	(K_1)_u\rightarrow \left(K_1\times_G K_2\right)_{(u,v)}; k_1\mapsto \big( k_1 , {\varphi_v}\inv\big(\psi(k_1)\big) \big)
	\end{align*}
	is the inverse map of $\tilde{\varphi}_{(u,v)}$ for every $(u,v)\in \left(K_1\times_G K_2\right)\0$.
\end{proof}

\begin{lemm}\label{Lemma: pull-back psi0 inj}
	If $\psi\0$ is injective, then so is $\widetilde{\psi}\0$.
\end{lemm}
\begin{proof}
	This follows from a simple calculation.
\end{proof}

Assume that $G,K_1,K_2$ are \'etale groupoids, and that $\psi\: K_1\rightarrow G$, $\varphi\:K_2\rightarrow G$ are continuous.
The groupoid $K_1\times_{G} K_2$ is \'etale with respect to the product topology.
The groupoid homomorphisms $\widetilde{\varphi}\: K_1\times_{G} K_2\rightarrow K_1$ and $\widetilde{\psi}\: K_1\times_{G} K_2\rightarrow K_2$ are continuous.

\begin{lemm}\label{Lemma: pull-back psi0 open}
	If $\psi\0$ is open, then so is $\widetilde{\psi}\0$.
\end{lemm}
\begin{proof}
	For open subsets $U \subset K_1\0$ and $V\subset K_2\0$, we get
	\[
	\widetilde{\psi}\0\big( (U\times V) \cap (K_1\0\times_G K_2\0) \big) = \big(\varphi\0\big)\inv\big( \psi\0(U) \big) \cap V
	\]
	by a simple calculation.
	This subset is open because $\psi\0$ is open.
	Hence $\widetilde{\psi}\0$ is open.
\end{proof}

Assume that $\varphi$ is fibrewise bijective, and that $\psi$ has open injective $\psi\0$.
By Lemma \ref{Lemma: pull-back fibrewise}, \ref{Lemma: pull-back psi0 inj}, and \ref{Lemma: pull-back psi0 open}, $\widetilde{\psi}$ has open injective $\widetilde{\psi}\0$, and $\widetilde{\varphi}$ is fibrewise bijective.
For an open bisection $U$ of $K_1$, $\tilde{\varphi}\inv(U)$ is an open bisection by Lemma \ref{Lemma: varphi inv}.
The maps $\psi|_U$ and $\widetilde{\psi}|_{\widetilde{\varphi}\inv(U)}$ are homeomorphisms onto their images by Lemma \ref{Lemma: psi 2}.

\begin{lemm}\label{Lemma: pull-back 2}
	For an open bisection $U$ of $K_1$,
	\[
	\psi|_U\inv\circ\varphi = \widetilde{\varphi}\circ\widetilde{\psi}|_{\widetilde{\varphi}\inv(U)}\inv
	\]
	holds as partial maps from $K_2$ to $K_1$.
\end{lemm}
\begin{proof}
	The domain of $\psi|_U\inv\circ\varphi$ is $\varphi\inv(\psi(U))$, and that of $\tilde{\varphi}\circ\tilde{\psi}|_{\tilde{\varphi}\inv(U)}\inv$ is $\tilde{\psi}(\tilde{\varphi}\inv(U))$. 
	They are the same open bisection of $K_2$ by Lemma \ref{Lemma: pull-back 1}.
	For every $k_2 \in \varphi\inv(\psi(U))$, we set $k_1 := \psi|_U\inv(\varphi(k_2)) \in U$.
	The pair $(k_1,k_2)$ belongs to $\tilde{\varphi}\inv(U)$, and $\tilde{\psi}(k_1,k_2) = k_2$ holds.
	Thus we get 
	\[
	\tilde{\varphi} \big(\tilde{\psi}|_{\tilde{\varphi}\inv(U)}\inv(k_2) \big) = \tilde{\varphi}(k_1,k_2) = k_1. \qedhere
	\]
\end{proof}

\begin{lemm}\label{Lemma: pull-back phi0 proper}
	If $\varphi\0$ is proper, then so is $\widetilde{\varphi}\0$.
\end{lemm}
\begin{proof}
	Take a compact subset $C$ of $K_1\0$.
	Since $\widetilde{\psi}\0$ is open and injective, $\widetilde{\psi}\0$ is a homeomorphism onto the image. 
	Hence $\big(\widetilde{\varphi}\0\big)\inv(C)$ is homeomorphic to $\widetilde{\psi}\0\big( \big(\widetilde{\varphi}\0\big)\inv(C)\big)$ 
	which coincides with $\big(\varphi\0\big)\inv(\psi\0(C))$ by Lemma \ref{Lemma: pull-back 1} (see also Lemma \ref{Lemma: varphi inv} (ii)).
	This subset of $K_2\0$ is compact because $\varphi\0$ is proper.
	Thus the subset $\big(\widetilde{\varphi}\0\big)\inv(C)$ of $(K_1\times_G K_2)\0$ is compact.
\end{proof}


Now we define composition of \egmors.

\begin{prop}\label{Proposition: comp cpl mor}
	Let $G_i$ be an \'etale groupoid with $i=1,2,3$, and $(\varphi_i,\psi_i;K_i)$ be a \egmor from $G_i$ to $G_{i+1}$ with $i=1,2$.
	The triplet 
	\[
	\big( \varphi_1\circ\widetilde{\varphi_2}, \psi_2\circ\widetilde{\psi_1}; K_1\times_{G_2} K_2 \big)
	\]
	is a \egmor from $G_1$ to $G_3$;
	\[
	\begin{tikzcd}[ampersand replacement=\&]
	\&\& K_1 \times_{G_2} K_2 \arrow[ld,"\widetilde{\varphi_2}"']\arrow[rd,"\widetilde{\psi_1}"] \&\&\\
	\& K_1 \arrow[ld,"\varphi_1"']\arrow[rd,"\psi_1"] \&\& K_2\arrow[ld,"\varphi_2"']\arrow[rd,"\psi_2"] \&\\
	G_1\&\& G_2 \&\& G_3.
	\end{tikzcd}
	\]
	If $( \varphi_i, \psi_i; K_i )$ is equivalent to a \egmor $( \varphi'_i, \psi'_i; K'_i )$ for $i=1,2$, then $\big(\varphi_1\circ\widetilde{\varphi_2}, \psi_2\circ\widetilde{\psi_1}; K_1\times_{G_2} K_2 \big)$ is equivalent to $\big(\varphi'_1\circ\widetilde{\varphi'_2}, \psi'_2\circ\widetilde{\psi'_1}; K'_1\times_{G_2} K'_2 \big)$.
\end{prop}
\begin{proof}
	For given \egmors $(\varphi_i,\psi_i;K_i)$ from $G_i$ to $G_{i+1}$, the triplet $(\varphi_1\circ\widetilde{\varphi_2},\psi_2\circ\widetilde{\psi_1};K_1\times_{G_2} K_2)$ is a \egmor by 
	Lemma \ref{Lemma: pull-back fibrewise}, \ref{Lemma: pull-back psi0 inj}, \ref{Lemma: pull-back psi0 open} and Remark \ref{Remark: composition fib bij}, \ref{Remark: comp psi0 open inj}.
	If $\iota_i$ is an equivalence from $(\varphi_i,\psi_i;K_i)$ to $(\varphi'_i,\psi'_i;K'_i)$ for $i=1,2$, then the map $\iota_1\times\iota_2$ from $K_1 \times_{G_2} K_2$ to $K'_1 \times_{G_2} K'_2$ defined as 
	\[
	(\iota_1\times\iota_2)(k_1,k_2):= (\iota_1(k_1),\iota_2(k_2))
	\]
	for $(k_1,k_2)\in K_1 \times_{G_2} K_2$ is an equivalence from $(\varphi_1\circ\widetilde{\varphi_2},\psi_2\circ\widetilde{\psi_1};K_1\times_{G_2} K_2)$ to $(\varphi'_1\circ\widetilde{\varphi'_2},\psi'_2\circ\widetilde{\psi'_1};K'_1\times_{G_2} K'_2)$.
\end{proof}

\begin{defi}\label{Definition: composition of couple morphisms}
	Let $G_i$ be an \'etale groupoid with $i = 1,2,3$, and $\kappa_i = [\varphi_i,\psi_i;K_i]\: G_i \rightarrow G_{i+1}$ be a \egmor with $i = 1,2$.
	We define composition of $\kappa_1$ and $\kappa_2$ as 
	\[
	\kappa_2\cdot \kappa_1 := \big[\varphi_1\circ\widetilde{\varphi_2},\psi_2\circ\widetilde{\psi_1};K_1\times_{G_2} K_2\big] \: G_1 \rightarrow G_3.
	\]
	This definition is well-defined by Proposition \ref{Proposition: comp cpl mor}.
\end{defi}

\begin{lemm}\label{Lemma: associativity of pull-back}
	The composition of \egmors satisfies the associative law, that is, $\kappa_3\cdot(\kappa_2\cdot\kappa_1)=(\kappa_3\cdot\kappa_2)\cdot\kappa_1$ holds for an \'etale groupoid $G_i$ with $i = 1,2,3,4$ and a \egmor $\kappa_i=[\varphi_i,\psi_i;K_i]\: G_i\rightarrow G_{i+1}$ with $i=1,2,3$.
\end{lemm}
\begin{proof}
	We set
	\begin{align*}
	K &:= \big(K_1\times_{G_2} K_2\big)\times_{G_3} K_3,\\
	\varphi&:((k_1,k_2),k_3) \mapsto (k_1,k_2) \mapsto k_1 \mapsto \varphi_1(k_1),\\
	\psi&:((k_1,k_2),k_3) \mapsto k_3 \mapsto \psi_3(k_3),
	\end{align*}
	and 
	\begin{align*}
	K' &:= K_1\times_{G_2} \big(K_2\times_{G_3} K_3\big),\\
	\varphi'&:(k_1,(k_2,k_3)) \mapsto k_1 \mapsto \varphi_1(k_1),\\
	\psi'&:(k_1,(k_2,k_3)) \mapsto (k_2,k_3) \mapsto k_3 \mapsto \psi_3(k_3).
	\end{align*}
	We have $\kappa_3\cdot(\kappa_2\cdot\kappa_1) = [\varphi,\psi;K]$ and $(\kappa_3\cdot\kappa_2)\cdot\kappa_1 = [\varphi',\psi';K']$;
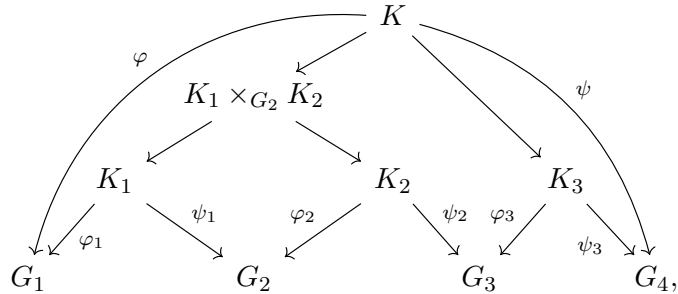
	\[
	\begin{tikzcd}
	&[-6mm]&[-6mm]&[-6mm]&[-6mm]&[-6mm]&[-6mm]\\[-6mm]
	&&&
	K \arrow[dl,""] \arrow[ddrr,""] 
	\arrow[dddlll,"\varphi"'{name = phi},bend right = 40]
	\arrow[dddrrr,"\psi",bend left = 30] 
	&&&\\[-2mm]
	&&
	K_1\times_{G_2}K_2 \arrow[dl,""] \arrow[dr,""] &&&&\\[-2mm]
	&
	K_1 \arrow[dl,"\varphi_1"] \arrow[dr,"\psi_1"] &&
	K_2 \arrow[dl,"\varphi_2"'] \arrow[dr,"\psi_2"] &&
	K_3 \arrow[dl,"\varphi_3"'] \arrow[dr,"\psi_3"'] &\\
	G_1&&G_2&&G_3&&G_4,
	%
	\end{tikzcd}
	\]
	
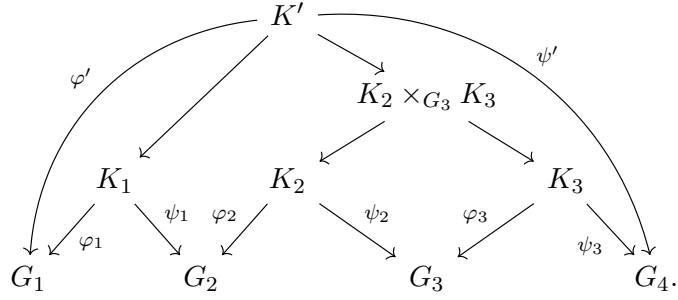
\[
	\begin{tikzcd}
	&[-6mm]&[-6mm]&[-6mm]&[-6mm]&[-6mm]&[-6mm]\\[-6mm]
	&&&
	K' \arrow[ddll,""] \arrow[dr,""] 
	\arrow[dddlll,"\varphi' "',bend right = 40] 
	\arrow[dddrrr,"\psi' ",bend left = 40] 
	&&&\\[-2mm]
	&&&&
	K_2\times_{G_3}K_3 \arrow[dl,""] \arrow[dr,""] &&\\[-2mm]
	&
	K_1 \arrow[dl,"\varphi_1"] \arrow[dr,"\psi_1"] &&
	K_2 \arrow[dl,"\varphi_2"'] \arrow[dr,"\psi_2"] &&
	K_3 \arrow[dl,"\varphi_3"'] \arrow[dr,"\psi_3"'] &\\
	G_1&&G_2&&G_3&&G_4.
	\end{tikzcd}
	\]
	We define $\iota$ as a homeomorphic groupoid isomorphism as
	\[
	\iota\: K \rightarrow K';((k_1,k_2),k_3) \mapsto (k_1,(k_2,k_3)).
	\]
	We can check easily that $\varphi = \varphi'\circ\iota$ and $\psi = \psi'\circ\iota$.
	Thus $\iota$ is an equivalence between $\kappa_3\cdot(\kappa_2\cdot\kappa_1)$ and $(\kappa_3\cdot\kappa_2)\cdot\kappa_1$.
\end{proof}

\begin{lemm}\label{Lemma: unitor}
	Let $G,H$ be \'etale groupoids, and $[\varphi,\psi;K]\: G \rightarrow H$ be a \egmor.
	\begin{enumerate}[(i)]
		\item For an \'etale groupoid $G'$ and a continuous fibrewise bijective groupoid homomorphism $\varphi'\: G \rightarrow G'$, the \egmor $[\varphi',\id_G;G]\: G' \rightarrow G$ satisfies $[\varphi,\psi;K] \cdot [\varphi',\id_G;G] = [\varphi'\circ\varphi,\psi;K]$.
		\item For an \'etale groupoid $H'$ and a continuous groupoid homomorphism $\psi'\: H \rightarrow H'$ with open injective $\psi\0$, the \egmor $[\id_H,\psi';H]\: H \rightarrow H'$ satisfies $[\id_H,\psi';H] \cdot [\varphi,\psi;K] = [\varphi,\psi'\circ\psi;K]$.
	\end{enumerate}
\end{lemm}
\begin{proof}
	\begin{enumerate}[(i)]
		\item The homeomorphic groupoid isomorphism $\widetilde{\id_G}\: G\times_G K \rightarrow K$ is an equivalence from $\big( \varphi'\circ\widetilde{\varphi}, \psi\circ\widetilde{\id_G}; G \times_G K \big)$ to $(\varphi'\circ\varphi,\psi;K)$.
		\item The homeomorphic groupoid isomorphism $\widetilde{\id_H}\: K \times_H H \rightarrow H$ is an equivalence from $\big( \varphi\circ\widetilde{\id_H}, \psi'\circ\widetilde{\psi}; K\times_H H\big)$ to $(\varphi,\psi'\circ\psi;K)$. \qedhere
	\end{enumerate}
\end{proof}

\begin{theo}\label{Theorem: category EG}
	All \'etale groupoids and all (equivalence classes of) \egmors form a category with the composition defined in Definition \ref{Definition: composition of couple morphisms} and the identity morphisms $[\id_G,\id_G;G]$ for all $G$.
	We denote this category as $\EG$.
\end{theo}
\begin{proof}
	The associative law of composition holds by Lemma \ref{Lemma: associativity of pull-back}. 
	The unit law of identity follows by Lemma \ref{Lemma: unitor}.
\end{proof}

Let $G,H$ be \'etale groupoids.

\begin{lemm}\label{Lemma: unitor 2}
	Let $[\varphi,\psi;K]\: G \rightarrow H$ be a \egmor.
	If $\varphi\: K \rightarrow G$ is a homeomorphic groupoid isomorphism, then we get $[\varphi,\psi;K] = [\id_G,\psi\circ\varphi\inv;G]$.
	If $\psi\: K \rightarrow H$ is a homeomorphic groupoid isomorphism, then we get $[\varphi,\psi;K] = [\varphi\circ\psi\inv,\id_H;H]$.
\end{lemm}
\begin{proof}
	A homeomorphic groupoid isomorphism $\varphi\: K \rightarrow G$ is an equivalence from $(\varphi,\psi;K)$ to $(\id_G,\psi\circ\varphi\inv;G)$.
	A homeomorphic groupoid isomorphism $\psi\: K \rightarrow H$ is an equivalence from $(\varphi,\psi;K)$ to $(\varphi\circ\psi\inv,\id_H;H)$.
\end{proof}

\begin{lemm}\label{Lemma: unitor 3}
	Let $K$ be an \'etale groupoid.
	If $\varphi\: K \rightarrow G$ and $\psi\: K \rightarrow H$ are homeomorphic groupoid isomorphisms, then $[\psi,\varphi;K]$ is the inverse of $[\varphi,\psi;K]$.
	Thus $[\varphi,\psi;K]$ is an isomorphism in $\EG$.
\end{lemm}
\begin{proof}
	By Lemma \ref{Lemma: unitor} and \ref{Lemma: unitor 2}, we get 
	\[
	[\psi,\varphi;K] \cdot [\varphi, \psi; K] = [\id_H,\varphi\circ\psi\inv;H] \cdot [\id_G,\psi\circ\varphi\inv;G] = [\id_G,\id_G;G]. 
	\]
	Similarly, we get $[\varphi, \psi; K] \cdot [\psi,\varphi;K] = [\id_H,\id_H;H]$.
	Thus $[\varphi,\psi;K]$ is an isomorphism in $\EG$.
\end{proof}

\begin{prop}
	A \egmor $\kappa= [\varphi, \psi; K] \: G \rightarrow H$ is an isomorphism in $\EG$ if and only if $\varphi$ and $\psi$ are homeomorphic groupoid isomorphisms.
\end{prop}
\begin{proof}
	The if part follows from Lemma \ref{Lemma: unitor 3}.
	
	We show that the only if part.
	Let $\kappa'=[\varphi',\psi';K']$ be the inverse morphism of $\kappa$.
	The following two diagrams commute;
	\[
	\begin{tikzcd}[ampersand replacement=\&]
		\&[-5mm]\&[-5mm] G\arrow[d,"\simeq"{sloped},phantom,"\tiny\exists"{xshift = -0.7em}] 
		\arrow[dddll,"\id_G"',bend right = 40]\arrow[dddrr,"\id_G",bend left = 40]
		\&[-5mm]\&[-5mm]\\[-5mm]
		\&\& K \times_{H} K' \arrow[ld,"\widetilde{\varphi'}"']
		\arrow[rd,"\widetilde{\psi}"]
		\&\&\\
		\& K \arrow[ld,"\varphi"']\arrow[rd,"\psi"] \&\& K'\arrow[ld,"\varphi'"']\arrow[rd,"\psi'"] \&\\
		G\&\& H \&\& G,
	\end{tikzcd}
	\begin{tikzcd}[ampersand replacement=\&]
		\&[-5mm]\&[-5mm] H\arrow[d,"\simeq"{sloped},phantom,"\tiny\exists"{xshift = -0.7em}]
		\arrow[dddll,"\id_H"',bend right = 40]\arrow[dddrr,"\id_H",bend left = 40]
		\&[-5mm]\&[-5mm]\\[-5mm]
		\&\& K' \times_{G} K \arrow[ld,"\widetilde{\varphi}"']
		\arrow[rd,"\widetilde{\psi'}"] 
		\&\&\\
		\& K' \arrow[ld,"\varphi'"']\arrow[rd,"\psi'"] \&\& K \arrow[ld,"\varphi"']\arrow[rd,"\psi"] \&\\
		H \&\& G \&\& H.
	\end{tikzcd}
	\]
	These diagrams imply that the ``non-tildes'' $\varphi,\psi,\varphi',\psi'$ are surjective and the ``tildes'' $\widetilde{\varphi},\widetilde{\psi},\widetilde{\varphi'},\widetilde{\psi'}$ are injective.
	By simple calculations, the ``tildes'' are surjective because so is ``non-tildes''. 
	Thus all homomorphisms in the above are continuous groupoid isomorphisms.
	Hence these are homeomorphic groupoid isomorphisms.
\end{proof}

\begin{defi}
	Let $G,H$ be \'etale groupoids.
	A \egmor $[\varphi,\psi]\: G \rightarrow H$ is said to be \emph{proper} if $\varphi\0$ is proper.
\end{defi}
We can check that this definition is well-defined, that is, for equivalent two \egmors $(\varphi_1,\psi_1;K_1)$ and $(\varphi_2,\psi_2;K_2)$, one is proper if and only if so is the other.

When constructing $C^*$-algebras and $*$-homomorphisms, we assume that \'etale groupoids have locally compact Hausdorff unit spaces, and that all \egmors between them are proper.

\begin{theo}\label{Theorem: category EGlcH}
	All \'etale groupoids with locally compact Hausdorff unit spaces and all proper \egmors form a subcategory of $\EG$.
	We denote this subcategory as $\EGlcH$.
\end{theo}
\begin{proof}
	It is clear that identity \egmors are proper. 
	By Lemma \ref{Lemma: pull-back phi0 proper}, the composition of proper \egmors is proper.
\end{proof}

\section{Functors $\TG$ and $\SA$}
\label{Section: functors TG and SA}

\subsection{A functor $\TG$ from $\ISA$ to $\EG$}
\label{Subsection: A functor TG from ISA to EG}

In this subsection, we construct a functor $\TG$ from $\ISA$ to $\EG$. 
Let $(S,X,\alpha)$ and $(T,Y,\beta)$ be inverse semigroup actions and $(\theta,\xi)$ be an \isamor from $(S,X,\alpha)$ to $(T,Y,\beta)$.
We define an inverse semigroup action $(S,D_\xi,\beta\theta)$ as $D^{\beta\theta}_s = D^\beta_{\theta(s)}$ and $(\beta\theta)_s = \beta_{\theta(s)}$ for every $s \in S$.
We recall that $\bigcup_{s \in S} D^{\beta\theta}_s = D_\xi$ by Remark \ref{Remark: isamor}.
We consider the transformation groupoid $S \ltimes_{\beta\theta} D_\xi$ of the inverse semigroup action $(S,D_\xi,\beta\theta)$.
We define a map $\varphi_\xi \: S \ltimes_{\beta\theta} D_\xi \rightarrow S\ltimes_\alpha X$ as
\begin{align*}
\varphi_\xi( [s,y] ) := [s,\xi(y)],
\end{align*}
and a map $\psi_\theta \: S \ltimes_{\beta\theta} D_\xi \rightarrow T\ltimes_\beta Y$ as
\begin{align*}
\psi_\theta( [s,y] ) := [\theta(s),y]
\end{align*}
for every $[s,y] \in S\ltimes_{\beta\theta} D_\xi$.

\begin{prop}\label{Proposition: TG morphism}
	The couple $(\varphi_\xi,\psi_\theta)$ becomes a \egmor from $S\ltimes_\alpha X$ to $T\ltimes_\beta Y$.
\end{prop}
\begin{proof}
	We first check that $\varphi_\xi$ and $\psi_\theta$ are well-defined.
	For every $[s,y] \in S \ltimes_{\beta\theta} D_\xi$, we get $y \in D^{\beta\theta}_s = D^\beta_{\theta(s)} = \xi\inv(D^\alpha_s)$.
	Thus there exist elements in the form of $[s,\xi(y)]$ in $S \ltimes_\alpha X$ and in the form of $[\theta(s),y]$ in $T \ltimes_\beta Y$.
	
	Take $[s,y],[s',y']\in S\ltimes_{\beta\theta} D_\xi$ with $[s,y] = [s',y']$.
	We get $y = y'$ and an element $e\in E(S)$ with $y=y' \in D^{\beta\theta}_e$ and $se = s'e$.
	This implies that $\xi(y) = \xi(y')\in D_e^\alpha$. 
	Thus $[s,\xi(y)] = [s',\xi(y')]$ holds. 
	Hence $\varphi_\xi$ is well-defined. 
	We also get $y = y' \in D^\beta_{\theta(e)}$ and $\theta(s)\theta(e) = \theta(se) = \theta(s'e) = \theta(s')\theta(e)$.
	Thus $[\theta(s),y] = [\theta(s'),y']$ holds.
	Hence $\psi_\theta$ is well-defined. 
	
	It is easy to check that $\varphi_\xi$ and $\psi_\theta$ are continuous, and $\psi_\theta$ is a groupoid homomorphism.
	If a pair $[s_2,y_2], [s_1,y_1] \in S\ltimes_{\beta\theta} D_\xi$ is composable, then we get $y_2 = \beta_{\theta(s_1)}(y_1)$.
	This implies $\xi(y_2) = \xi(\beta_{\theta(s_1)}(y_1)) = \alpha_{s_1}(\xi(y_1))$. 
	Thus the pair of $[s_2,\xi(y_2)]$ and $[s_1,\xi(y_1)]$ is composable.
	We can easily check that $\varphi_\xi$ keeps multiplication.
	Thus $\varphi_\xi$ is a groupoid homomorphism.
	
	We show that $\varphi_\xi$ is fibrewise bijective.
	Take $y \in D_\xi$.
	We can check easily that the map $(\varphi_\xi)_y$ is surjective.
	Take $[s,y], [s',y] \in (S\ltimes_{\beta\theta} D_\xi)_y$ with $[s,\xi(y)] = [s',\xi(y)]$.
	There exists $e\in E(S)$ with $\xi(y)\in D^\alpha_e$ and $se = s'e$. 
	For such $e\in E(S)$, we get $y\in \xi\inv(D_e^\alpha) = D^{\beta\theta}_e$.
	This implies $[s,y] = [s',y]$. 
	Thus $(\varphi_\xi)_y$ is injective.
	
	It is clear that the map $\psi_\theta\0\: (S\ltimes_{\beta\theta} D_\xi)\0 \simeq D_\xi \hookrightarrow Y \simeq (T\ltimes_\beta Y)\0$ is injective and open.
\end{proof}

\begin{theo}\label{Theorem: functor TG}
	 The constructions $(S,X,\alpha)\mapsto S\ltimes_\alpha X$ and $(\theta,\xi)\mapsto [\varphi_\xi,\psi_\theta]$ form a functor from $\ISA$ to $\EG$.
	 We denote this functor as $\TG$.
\end{theo}
\begin{proof}
	These constructions keep identity morphisms clearly.
	
	Let $(S_i,X_i,\alpha_i)$ be an inverse semigroup action with $i=1,2,3$, and $(\theta_i,\xi_i)$ be an \isamor from $(S_i,X_i,\alpha_i)$ to $(S_{i+1},X_{i+1},\alpha_{i+1})$ with $i=1,2$. 
	We set 
	\[
	\begin{aligned}
	K&:= ( S_1 \ltimes_{\alpha_2\theta_1} D_{\xi_1} ) \times_{S_2\ltimes_{\alpha_2} X_2} ( S_2\ltimes_{\alpha_3\theta_2} D_{\xi_2} ) \text{ and }\\
	K'&:=S_1\ltimes_{\alpha_3(\theta_2\circ\theta_1)} D_{\xi_1\circ\xi_2}.
	\end{aligned}
	\]
	We have to show that the \egmors $\big[\varphi_{\xi_2},\psi_{\theta_2}\big]\cdot\big[\varphi_{\xi_1},\psi_{\theta_1}\big]$ and $\big[\varphi_{\xi_1\circ\xi_2},\psi_{\theta_2\circ\theta_1}\big]$ coincide:
	\[
	\begin{tikzcd}
	( S_1, \arrow[dd,"\theta_1"'] &[-12mm] X_1, &[-12mm] \alpha_1 ) \arrow[r,"\mapsto",phantom] &[5mm] 
	S_1\ltimes_{\alpha_1} X_1&&&[-5mm]\\[-5mm]
	&&&&
	S_1\ltimes_{\alpha_2\theta_1} D_{\xi_1} \arrow[lu,"\varphi_{\xi_1}"']\arrow[ld,"\psi_{\theta_1}"'] &&\\[-5mm]
	( S_2, \arrow[dd,"\theta_2"'] & X_2, \arrow[uu,"\xi_1"'] & \alpha_2 )\arrow[r,"\mapsto",phantom] & 
	S_2\ltimes_{\alpha_2} X_2&&
	K \arrow[lu,"\widetilde{\varphi_{\xi_2}}"']\arrow[ld,"\widetilde{\psi_{\theta_1}}"] \arrow[r,dashed] &
	K'. \arrow[llluu,"\varphi_{\xi_1\circ\xi_2}"',bend right = 20]\arrow[llldd,"\psi_{\theta_2\circ\theta_1}",bend left = 20]
	\\[-5mm]
	&&&&
	S_2\ltimes_{\alpha_3\theta_2} D_{\xi_2} \arrow[lu,"\varphi_{\xi_2}"]\arrow[ld,"\psi_{\theta_2}"] &&\\[-5mm]
	( S_3, & X_3, \arrow[uu,"\xi_2"'] & \alpha_3 )\arrow[r,"\mapsto",phantom] & 
	S_3\ltimes_{\alpha_3} X_3&&&
	\end{tikzcd}
	\]
	
	We define a map $\iota\: K\rightarrow K'$ as $\iota( [s_1,x_2] , [s_2,x_3] ) := [s_1,x_3]$.	
	We show that $\iota$ is well-defined.
	Take an element $( [s_1,x_2] , [s_2,x_3] ) \in K$. 
	Since $[\theta_1(s_1),x_2] = [s_2,\xi_2(x_3)]$, we get $x_2 = \xi_2(x_3)$. 
	Since $x_2\in D^{\alpha_2\theta_1}_{s_1} = \xi_1\inv(D^{\alpha_1}_{s_1})$, we get $\xi_1(\xi_2(x_3)) = \xi_1(x_2) \in D_{s_1}^{\alpha_1}$.
	This implies $x_3 \in (\xi_1\circ \xi_2)\inv(D^{\alpha_1}_{s_1}) = D^{\alpha_3(\theta_2\circ\theta_1)}_{s_1}$.
	Hence there exists an element in the form of $[s_1,x_3]$ in $K'$.
	Take $[s_i,x_{i+1}],[s'_i,x'_{i+1}] \in S_i \ltimes_{\alpha_{i+1}\theta_i} D_{\xi_i}$ for $i =1,2$ with $( [s_1,x_2] , [s_2,x_3] ) = ( [s_1',x_2'] , [s_2',x_3'] ) \in K$.
	We get $x_2 = \xi_2(x_3)$, $x_3 = x_3'$, and $e_1\in E(S_1)$ with $s_1e_1 = s_1'e_1$ and $x_2 = x_2' \in D^{\alpha_2\theta_1}_{e_1}$.
	We can check that these imply $x_3 = x_3' \in D^{\alpha_3(\theta_2\circ\theta_1)}_{e_1}$.
	Thus we get $[s_1,x_3] = [s'_1,x'_3]$.
	
	It is easy to check that $\varphi_{\xi_1} \circ \widetilde{\varphi_{\xi_2}} = \varphi_{\xi_1\circ\xi_2} \circ \iota $. 
	We show that $\psi_{\theta_2} \circ \widetilde{\psi_{\theta_1}} = \psi_{\theta_2\circ\theta_1} \circ \iota$.
	Take an element $\big( [s_1,x_2] , [s_2,x_3] \big) \in K$.
	We get
	\[
	\begin{aligned}
	\big(\psi_{\theta_2} \circ \widetilde{\psi_{\theta_1}}\big) \big( [s_1,x_2] , [s_2,x_3] \big) &= [\theta_2(s_2),x_3] \text{ and }\\
	(\psi_{\theta_2\circ\theta_1} \circ \iota) \big( [s_1,x_2] , [s_2,x_3] \big) &= [\theta_2(\theta_1(s_1)),x_3].
	\end{aligned}
	\]
	Since $[\theta_1(s_1),x_2] = [s_2,\xi_2(x_3)]$, there exists $e_2\in E(S_2)$ with $\xi_2(x_3) \in D^{\alpha_2}_{e_2}$ and $\theta_1(s_1)e_2 = s_2e_2$.
	We get $x_3\in \xi_2\inv\big(D_{e_2}^{\alpha_2}\big) = D_{\theta_2(e_2)}^{\alpha_3}$ and
	\[
	\theta_2(\theta_1(s_1))\theta_2(e_2) = \theta_2(\theta_1(s_1)e_2) = \theta_2(s_2e_2) =  \theta_2(s_2)\theta_2(e_2).
	\]
	Thus $[\theta_2(s_2),x_3] = [\theta_2(\theta_1(s_1)),x_3]$ holds.
	Hence we get $\psi_{\theta_2} \circ \widetilde{\psi_{\theta_1}} = \psi_{\theta_2\circ\theta_1} \circ \iota$.
	
	
	Now we have that the map $\iota$ satisfies $\varphi_{\xi_1} \circ \widetilde{\varphi_{\xi_2}} = \varphi_{\xi_1\circ\xi_2} \circ \iota $ and $\psi_{\theta_2} \circ \widetilde{\psi_{\theta_1}} = \psi_{\theta_2\circ\theta_1} \circ \iota$.
	One can see that $\iota(K\0)={K'}\0$.
	Thus we get $\big[\varphi_{\xi_2},\psi_{\theta_2}\big]\cdot\big[\varphi_{\xi_1},\psi_{\theta_1}\big]=\big[\varphi_{\xi_1\circ\xi_2},\psi_{\theta_2\circ\theta_1}\big]$ by Corollary \ref{Corollary: egmor equiv}.
\end{proof}

\begin{coro}
	The functor $\TG$ descends to a functor from $\ISAlcH$ to $\EGlcH$ which is also denoted as $\TG$.
\end{coro}
\begin{proof}
	Let $(S,X,\alpha)$, $(T,Y,\beta)$ be inverse semigroup actions on locally compact Hausdorff spaces, and $(\theta,\xi)\:(S,X,\alpha)\rightarrow(T,Y,\beta)$ be a proper \isamor.
	It is clear that the transformation groupoid $S \ltimes_\alpha X$ has a locally compact Hausdorff unit space.
	Since $\xi\: D_\xi \rightarrow X$ is proper, so is $\varphi_\xi\0 \: (S\ltimes_{\beta\theta} D_\xi)\0 \simeq D_\xi  \rightarrow X \simeq (S\ltimes_\alpha X)\0$.
	Hence the \egmor $(\varphi_\xi,\psi_\theta)$ is proper.
\end{proof}

The composition of the functors $\SP$ and $\TG$ sends inverse semigroups $S$ to their universal groupoids $\Gu(S)$ by their definitions.

\begin{defi}\label{Definition: functor Gu}
	We denote the composition of the functors $\SP\:\IS\rightarrow\ISAlcH$ and $\TG\:\ISAlcH\rightarrow\EGlcH$ as $\Gu$.
\end{defi}

The functor $\Gu$ sends semigroup homomorphisms $\theta$ to the \egmors $[\varphi_{\widehat{\theta}},\psi_\theta]$.


\subsection{A functor $\SA$ from $\EG$ to $\ISA$}

We construct a functor from $\EG$ to $\ISA$.
\begin{defi}[{\cite[Proposition 5.3]{Exe08}}]
	For an \'etale groupoid $G$, the \emph{slice action} $\big(\Bis G, G\0,\gamma^G\big)$ of $G$ is constructed as follows:
	For every $U \in \Bis G$, we set the domain of $\gamma^G_U$ as $d(U)$ and the partial homeomorphism $\gamma^G_U$ as $r|_U\circ\big(d|_U\big)\inv$.
\end{defi}

Let $G,H$ be \'etale groupoids.
We will see that a \egmor from $G$ to $H$ produces an \isamor from $\big(\Bis G, G\0,\gamma^G\big)$ to $\big(\Bis H, H\0,\gamma^H\big)$.

Let $K$ be an \'etale groupoid, and 
$\varphi\: K \rightarrow G$ be a continuous fibrewise bijective groupoid homomorphism.
Recall that the map $\varphi\inv\:\Bis G\rightarrow \Bis K; U \mapsto \varphi\inv(U)$ becomes a semigroup homomorphism (see Definition \ref{Definition: varphi inv}).

\begin{lemm}\label{Lemma: typeS isamor}
	The pair $(\varphi\inv,\varphi\0)$  becomes an \isamor from $\big(\Bis G, G\0, \gamma^G\big)$ to $\big(\Bis K, K\0, \gamma^K\big)$.
\end{lemm}
\begin{proof}
	For every $U\in\Bis G$, we have 
	\[
	D^{\gamma^K}_{\varphi\inv(U)} = d(\varphi\inv(U)) = (\varphi\0)\inv(d(U)) = \big(\varphi\0\big)\inv\big(D^{\gamma^G}_U\big)
	\]
	by Lemma \ref{Lemma: varphi inv} (ii) and (iii).
	For every $U \in \Bis G$ and $v \in D^{\gamma^K}_{\varphi\inv(U)}$, we get a unique element $k \in \varphi\inv(U)$ with $d(k) = v$.
	The element $g = \varphi(k)$ is a unique element such that $g \in U$ and $d(g) = \varphi\0(v)$.
	This implies that $\varphi\0\big(\gamma^K_{\varphi\inv(U)}(v)\big) = \varphi\0(r(k)) = r(g) = \gamma^G_U\big(\varphi\0(v)\big)$.
\end{proof}

Let $\psi\: K \rightarrow H$ be a continuous groupoid homomorphism with open injective $\psi\0$.
Recall that the map $\psi\: \Bis K \rightarrow \Bis H: U \mapsto \psi(U)$ becomes a semigroup homomorphism (see Definition \ref{Definition: psi}).
Since $\psi\0\:K\0\rightarrow H\0$ is a homeomorphism onto its open image $\psi\0(K\0)$, we can consider the partial map $\big(\psi\0\big)\inv\: H\0 \supset \psi\0(K\0) \rightarrow K\0$.

\begin{lemm}\label{Lemma: typeG isamor}
	The pair $\big(\psi,\big(\psi\0\big)\inv\big)$ becomes an \isamor from $(\Bis K, K\0, \gamma^K)$ to $(\Bis H, H\0, \gamma^H)$.
\end{lemm}
\begin{proof}
	We have 
	\[
	D^{\gamma^H}_{\psi(U)} = d(\psi(U)) = \psi\0(d(U)) = \psi\0\big(D^{\gamma^K}_U\big)
	\]
	for every $U\in\Bis K$.
	The right hand side is the pre-image of $D^{\gamma^K}_U$ by the partial map $(\psi\0)\inv$.
	For every $U \in \Bis K$ and $v \in D^{\gamma^H}_{\psi(U)}$, we get a unique element of $k \in U $ with $\psi\0(d(k)) = v$.
	The element $h = \psi(k)$ is the unique element such that $h \in \psi(U)$ and $d(h) = v$.
	This implies that $\gamma^K_U(\big(\psi\0)\inv(v)\big) = \psi\0(r(k)) = r(h) = (\psi\0)\inv\big(\gamma^H_{\psi(U)}(v)\big)$.
\end{proof}

\begin{prop}\label{Propositon: SA morphism}
	Let $(\varphi,\psi;K)$ be a \egmor from $G$ to $H$.
	The pair $\big(\psi\circ\varphi\inv, \varphi\0\circ\big(\psi\0\big)\inv\big)$ becomes an \isamor from $\big(\Bis G,G\0,\gamma^G\big)$ to $\big(\Bis H,H\0,\gamma^H\big)$.
\end{prop}
\begin{proof}
	The pair $\big(\psi\circ\varphi\inv, \varphi\0\circ\big(\psi\0\big)\inv\big)$ is obtained as the composition of \isamors in Lemma \ref{Lemma: typeS isamor} and \ref{Lemma: typeG isamor}.
\end{proof}

We remark that the domain of the partial map $\varphi\0\circ\big(\psi\0\big)\inv$ is the open subset $\psi\0\big(K\0\big)$ of $H\0$.

\begin{lemm}
	For equivalent \egmors $(\varphi_1,\psi_1;K_1)$ and $(\varphi_2,\psi_2;K_2)$ from $G$ to $H$, the \isamors associated with them as in Proposition \ref{Propositon: SA morphism} coincide.
\end{lemm}
\begin{proof}
	The semigroup homomorphisms $\psi_1\circ\varphi_1\inv$ and $\psi_2\circ\varphi_2\inv$ coincide by Corollary \ref{Corollary: egmor equiv}.
	We can see easily that the partial maps $\varphi_1\0\circ\big(\psi_1\0\big)\inv$ and $\varphi_2\0\circ\big(\psi_2\0\big)\inv$ coincide.
\end{proof}

\begin{theo}\label{Theorem: SA functoriality}
	The constructions $G\mapsto (\Bis G,G\0,\gamma^G)$ and $[\varphi,\psi] \mapsto \big(\psi\circ\varphi\inv, \varphi\0\circ\big(\psi\0\big)\inv\big)$ form a functor from $\EG$ to $\ISA$.
	We denote this functor as $\SA$.
\end{theo}
\begin{proof}
	This constructions keep identity morphisms clearly.
	
	Let $G_i$ be an \'etale groupoid with $i = 1,2,3$, and $[\varphi_i,\psi_i,K_i]\: G_i\rightarrow G_{i+1}$ be a \egmor with $i=1,2$. 
	We have to show that the following diagrams commute;
	\[
	\begin{tikzcd}[ampersand replacement=\&]
	\&[-14mm]\&[-14mm] 
	\Bis\left(K_1 \times_{G_2} K_2\right) \arrow[ld,"\widetilde{\varphi_2}\inv"',leftarrow]\arrow[rd,"\widetilde{\psi_1}"] 
	\&[-14mm]\&[-14mm]\\
	\& \Bis K_1 \arrow[ld,"\varphi_1\inv"',leftarrow]\arrow[rd,"\psi_1"] \&\& \Bis K_2 \arrow[ld,"\varphi_2\inv"',leftarrow]\arrow[rd,"\psi_2"] \&\\
	\Bis G_1\&\& \Bis G_2 \&\& \Bis G_3,
	\end{tikzcd}
	\begin{tikzcd}[ampersand replacement=\&]
	\&[-12mm]\&[-12mm] 
	(K_1 \times_{G_2} K_2)\0 \arrow[ld,"\widetilde{\varphi_2}\0"']\arrow[rd,"\big(\widetilde{\psi_1}\0\big)\inv"{yshift = -7pt},leftarrow]
	\&[-12mm]\&[-12mm]\\
	\& K_1\0 \arrow[ld,"\varphi_1\0"']\arrow[rd,"\big(\psi_1\0\big)\inv"{xshift = -5pt},leftarrow] \&\& K_2\0 \arrow[ld,"\varphi_2\0"]\arrow[rd,"\big(\psi_2\0\big)\inv",leftarrow] \&\\
	G_1\0 \&\& G_2\0 \&\& G_3\0.
	\end{tikzcd}
	\]
	The first diagram commutes by Lemma \ref{Lemma: pull-back 1}.

	By Lemma \ref{Lemma: pull-back fibrewise} and Lemma \ref{Lemma: varphi inv} (ii), we get 
	$
	\widetilde{\varphi_2}\inv\big(K_1\0\big) = (K_1 \times_G K_2)\0.
	$
	Hence we have 
	$
	 \big.\widetilde{\psi_1}\big|_{\widetilde{\varphi_2}\inv\big(K_1\0\big)} = \widetilde{\psi_1}\0.
	$
	Setting $U$ in Lemma \ref{Lemma: pull-back 2} as $K_1\0$, we get that the second diagram commutes.
\end{proof}

\begin{coro}
	The functor $\SA$ descends to a functor from $\EGlcH$ to $\ISAlcH$ which is also denoted as $\SA$.
\end{coro}
\begin{proof}
	It is trivial that $\SA$ sends an object in $\EGlcH$ to an object in $\ISAlcH$.
	
	Let $G,H$ be \'etale groupoids with locally compact Hausdorff unit spaces, and $[\varphi,\psi]\: G \rightarrow H$ be a proper \egmor.
	The partial map $\big(\psi\0\big)\inv$ is proper because the pre-image of a compact subset $C$ of $K\0$ by the partial map $\big(\psi\0\big)\inv$ coincides with the compact set $\psi\0(C)$.
	Thus the partial map $\varphi\0\circ\big(\psi\0\big)\inv$ is proper since $\varphi\0$ is proper.
	Hence $\big(\psi\circ\varphi\inv, \varphi\0\circ\big(\psi\0\big)\inv\big)$ is proper.
\end{proof}

\subsection{The functor $\TG$ is left adjoint to $\SA$}

We show that the functor $\TG$ from $\ISA$ to $\EG$ is left adjoint to the functor $\SA$ from $\EG$ to $\ISA$.
See \cite[Section IV]{Mac98} for basics of adjoint functors.

Let $(S,X,\alpha)$ be an inverse semigroup action.
We construct an \isamor $\eta_{(S,X,\alpha)}$ from $(S,X,\alpha)$ to $(\SA\cdot\TG)(S,X,\alpha)$. 
We remark that $(\SA\cdot\TG)(S,X,\alpha)$ is nothing but the inverse semigroup action
\[
\big( \Bis(S\ltimes_\alpha X), (S\ltimes_\alpha X)\0, \gamma^{S\ltimes_\alpha X} \big).
\]
We can check that the map 
\[
\theta_\alpha\: S\rightarrow \Bis(S\ltimes_\alpha X);s\mapsto[s,D^\alpha_{s}]
\]
becomes a semigroup homomorphism.
We set $\xi_\alpha$ as the homeomorphism $(S\ltimes_\alpha X)\0 \rightarrow X;[e,x]\mapsto x$.

\begin{lemm}\label{Lemma: eta_alpha}
	The pair $(\theta_\alpha,\xi_\alpha)$ becomes an \isamor from $(S,X,\alpha)$ to $(\SA\cdot\TG)(S,X,\alpha)$.
	We denote this \isamor as $\eta_{(S,X,\alpha)}$.
\end{lemm}
\begin{proof}
	Fix $s \in S$ arbitrarily. 
	We get  
	\[
	D^{\gamma^{S\ltimes_\alpha X}}_{\theta_\alpha(s)} 
	= D^{\gamma^{S\ltimes_\alpha X}}_{[s,D^\alpha_s]} 
	= d([s,D^\alpha_s]) 
	= [s^*s,D^\alpha_s]
	= \xi_\alpha\inv(D^\alpha_s).
	\]
	Take $[e,x]\in \xi_\alpha\inv(D^\alpha_s)$.
	We get $x = \xi_\alpha([e,x]) \in D_s^\alpha$.
	We have 
	\begin{align*}
	\xi_\alpha\big( \gamma^{S\ltimes_\alpha X}_{\theta_\alpha(s)}([e,x]) \big) 
	&= \xi_\alpha\big( \gamma^{S\ltimes_\alpha X}_{[s,D^\alpha_s]}([s^*s,x]) \big)\\
	&= \xi_\alpha([ss^*,\alpha_sx])
	= \alpha_sx
	= \alpha_s(\xi_\alpha([e,x])).
	\end{align*}
	Thus $(\theta_\alpha,\xi_\alpha)$ becomes an \isamor.
\end{proof}

\begin{lemm}\label{Lemma: eta}
	For every inverse semigroup actions $(S,X,\alpha)$, $(T,Y,\beta)$ and an \isamor $(\theta,\xi)$, the following diagram in $\ISA$ commutes;
	\[
	\begin{tikzcd}
	(S,X,\alpha) & (\SA\cdot\TG)(S,X,\alpha)\\
	(T,Y,\beta) & (\SA\cdot\TG)(T,Y,\beta).
	\arrow[from = 1-1, to = 1-2, "{\eta_{(S,X,\alpha)}}"] 
	\arrow[from = 2-1, to = 2-2, "{\eta_{(T,Y,\beta)}}"'] 
	\arrow[from = 1-1, to = 2-1, "{(\theta,\xi)}"']
	\arrow[from = 1-2, to = 2-2, "{(\SA\cdot\TG)(\theta,\xi)}"]
	\end{tikzcd}
	\]
	
\end{lemm}
\begin{proof}
	We have to show the following diagrams are commutative;
	\[
	\begin{tikzcd}
	S \arrow[r,"\theta_{\alpha}"] \arrow[d,"\theta"']  &\Bis\left( S\ltimes_{\alpha}X \right) \arrow[d,"\psi_\theta\circ\varphi_\xi\inv"] \\
	T \arrow[r,"\theta_{\beta}"'] &\Bis\left( T\ltimes_{\beta}Y \right),
	\end{tikzcd}
	\begin{tikzcd}
	X &\left( S\ltimes_{\alpha}X \right)\0 \arrow[l,"\xi_{\alpha}"'] \\
	Y \arrow[u,"\xi"] &\left( T\ltimes_{\beta}Y \right)\0. \arrow[u,"\varphi_\xi\0\circ\left(\psi_\theta\0\right)\inv"'] \arrow[l,"\xi_{\beta}"]
	\end{tikzcd}
	\]
	We get 
	\[
	\begin{aligned}
	(\psi_\theta\circ\varphi_\xi\inv\circ\theta_{\alpha})(s) 
	&= (\psi_\theta\circ\varphi_\xi\inv)([s,D^\alpha_s])\\
	&= \psi_\theta([s,\xi\inv(D^{\alpha}_s)])
	= [\theta(s),\xi\inv(D_s^{\alpha})],\\
	(\theta_{\beta}\circ\theta)(s) 
	&= \theta_{\beta}(\theta(s)) 
	= \big[\theta(s),D^{\beta}_{\theta(s)}\big]
	\end{aligned}
	\]
	for every $s \in S$.
	By the condition (i) in Definition \ref{Definition: isamor}, the first square is commutative.
	
	We show that the second square is commutative.
	The domain of the partial maps $\xi\circ\xi_\beta$ and $\xi_\alpha\circ(\varphi_\xi)\0\circ\big(\psi_\theta\0\big)\inv$ are $\xi\inv_\beta(D_\xi)$ and $\psi_\theta\0\big((S\ltimes_{\beta\theta}D_\xi)\0\big)$ respectively.
	It is clear that $\psi_\theta\0\big( (S\ltimes_{\beta\theta}D_\xi)\0 \big) \subset \xi\inv_\beta(D_\xi)$ holds.
	Take $[f,y]\in\xi_\beta\inv(D_\xi)$.
	We get $y \in D_\xi$.
	By Remark \ref{Remark: isamor}, there exists $e\in E(S)$ with $y \in D^\beta_{\theta(e)}$.
	Thus we get $[f,y] = [\theta(e),y] \in \psi_\theta\0\big( (S\ltimes_{\beta\theta}D_\xi)\0 \big)$.
	For every $[f,y] \in \xi_\beta\inv(D_\xi)$, the elements $(\xi\circ\xi_\beta)([f,y])$ and $\big(\xi_\alpha\circ\varphi_\xi\0\circ\big(\psi_\theta\0\big)\inv\big)([f,y])$ are the same element $\xi(y) \in X$.
	Thus the two partial maps coincide.
\end{proof}

This lemma means that the collection $\eta = \big\{\eta_{(S,X,\alpha)}\big\}$ forms a natural transformation from $1_{\ISA}$ to $\SA\cdot\TG$.

Let $G$ be an \'etale groupoid.
We construct a \egmor from $(\TG\cdot\SA)(G)$ to $G$.
We remark that the \'etale groupoid $(\TG\cdot\SA)(G)$ is nothing but the \'etale groupoid $\Bis G\ltimes_{\gamma^G} G\0$.
We define a map
\[
\omega_G\: \Bis G\ltimes_{\gamma^G} G\0 \rightarrow G;[U,u] \mapsto g_u,
\]
where $g_u$ is the unique element of $U$ with $d(g_u) = u$. 
This map becomes a homeomorphic groupoid isomorphism (see \cite[Proposition 5.4]{Exe08}). 
We denote as $\varepsilon_G$ 
the \egmor $[\id_{(\TG\cdot\SA)(G)}, \omega_G; (\TG\cdot\SA)(G)]$ from $(\TG\cdot\SA)(G)$ to $G$ 
which is an isomorphism in $\EG$ by Lemma \ref{Lemma: unitor 3}. 

\begin{lemm}\label{Lemma: eps}
	For every \'etale groupoids $G,H$ and a \egmor $\kappa\: G \rightarrow H$, the following diagram in $\EG$ commutes;
	\[
	\begin{tikzcd}
	(\TG\cdot\SA)(G) & G\\
	(\TG\cdot\SA)(H) & H.
	\arrow[from = 1-1, to = 1-2, "{\varepsilon_G}"] 
	\arrow[from = 2-1, to = 2-2, "{\varepsilon_H}"'] 
	\arrow[from = 1-1, to = 2-1, "{(\TG\cdot\SA)(\kappa)}"']
	\arrow[from = 1-2, to = 2-2, "{\kappa}"]
	\end{tikzcd}
	\]
\end{lemm}
\begin{proof}
	Let $G,H$ be \'etale groupoids, and $\kappa = [\varphi,\psi;K]\:G\rightarrow H$ be a \egmor.
	Since $[\varphi,\psi] = [\id_K,\psi] \cdot [\varphi,\id_K]$ holds, it suffices to show that the following two squares commute;
	\[
	\begin{tikzcd}
	(\TG\cdot\SA)(G) & G\\
	(\TG\cdot\SA)(K) & K\\
	(\TG\cdot\SA)(H) & H.
	\arrow[from = 1-1, to = 1-2, "{\varepsilon_G}"]
	\arrow[from = 2-1, to = 2-2, "{\varepsilon_K}"]
	\arrow[from = 3-1, to = 3-2, "{\varepsilon_H}"]
	\arrow[from = 1-1, to = 2-1, "{(\TG\cdot\SA)([\varphi,\id_K])}"']
	\arrow[from = 2-1, to = 3-1, "{(\TG\cdot\SA)([\id_K,\psi])}"']
	\arrow[from = 1-2, to = 2-2, "{[\varphi,\id_K]}"]
	\arrow[from = 2-2, to = 3-2, "{[\id_K,\psi]}"]
	\end{tikzcd}
	\]
	
	We first show that the first square commutes. 
	We denote by $(\theta,\xi)$ the \isamor $\SA([\varphi, \id_K]) = \big( \varphi\inv, \varphi\0 \big)$ from $\SA(G) = \big(\Bis(G),G\0,\gamma^G\big)$ to $\SA(K) = \big( \Bis(K), K\0, \gamma^K \big)$.
	The \egmor $(\TG\cdot\SA)([\varphi, \id_K])$ is nothing but the \egmor $[\varphi_\xi,\psi_\theta; \Bis G \ltimes_{\gamma^K\theta} K\0]$.
	We have $\varphi_\xi([U,v]) = [U,\varphi\0(v)]$ and $\psi_\theta([U,v])) = [\varphi\inv(U),v]$ for $[U,v] \in \Bis G \ltimes_{\gamma^K\theta} K\0$.
	By Lemma \ref{Lemma: unitor} and \ref{Lemma: unitor 3}, we get 
	\begin{align*}
	&\varepsilon_K \cdot (\TG\cdot\SA)([\varphi, \id_K]) \cdot \varepsilon_G\inv \\
	&= [\id_{(\TG\cdot\SA)(K)}, \omega_K] \cdot [\varphi_\xi, \psi_\theta] \cdot [\id_{(\TG\cdot\SA)(G)},\omega_G]\inv\\
	&= [\id_{(\TG\cdot\SA)(K)}, \omega_K] \cdot [\varphi_\xi, \psi_\theta] \cdot [\omega_G,\id_{(\TG\cdot\SA)(G)}]\\
	&= [\omega_G\circ\varphi_\xi, \omega_K\circ\psi_\theta].
	\end{align*}
	Let us show $[\omega_G\circ\varphi_\xi, \omega_K\circ\psi_\theta] = [\varphi, \id_K]$.
	The continuous groupoid homomorphism $\omega_K\circ\psi_\theta$ becomes an equivalence from $(\varepsilon_G\circ\varphi_\xi, \varepsilon_K\circ\psi_\theta; \Bis G \ltimes_{\gamma^K\theta} K\0)$ to $(\varphi,\id_K;K)$;
	\[
	\begin{tikzcd}
	&\Bis G \ltimes_{\gamma^K\theta} K\0&\\
	G&&K.\\
	&K&
	\arrow[from = 1-2, to = 2-1, "{\omega_G\circ\varphi_\xi}"']
	\arrow[from = 1-2, to = 2-3, "{\omega_K\circ\psi_\theta}"]
	\arrow[from = 3-2, to = 2-1, "{\varphi}"]
	\arrow[from = 3-2, to = 2-3, "{\id_K}"']
	\arrow[from = 1-2, to = 3-2, "{\omega_K\circ\psi_\theta}",dashed]
	\end{tikzcd}
	\]
	It is trivial that the right triangle commutes.
	Take $[U,v] \in \Bis G \ltimes_{\gamma^K\theta} K\0$.
	We set $k:= \omega_K(\psi_\theta([U,v])) = \omega_K([\varphi\inv(U),v])$.
	This element $k \in K$ satisfies that $k \in \varphi\inv(U)$ and $d(k) = v$.
	The element $\varphi(k) \in G$ satisfies $\varphi(k) \in U$ and $d(\varphi(k)) = \varphi\0(d(k)) = \varphi\0(v)$.
	This element $\varphi(k)$ is nothing but the element $\omega_G(\varphi_\xi([U,v])) = \omega_G([U,\varphi\0(v)])$.
	Hence we get that $\omega_G \circ \varphi_\xi = \varphi \circ (\omega_K \circ \psi_\theta)$.
	One can see that $(\omega_K\circ\psi_\theta)\0$ is surjective.
	By Proposition \ref{Proposition: egmors}, $\omega_K\circ\psi_\theta$ is an equivalence.
	Thus we get $\varepsilon_K \cdot (\TG\cdot\SA)([\varphi, \id_K]) \cdot \varepsilon_G\inv = [\varphi, \id_K]$.
	Hence the first square commutes.
	
	Next we show that the second square commutes.
	We denote by $(\theta',\xi')$ the \isamor $\SA([\id_K,\psi]) = \big(\psi,\big(\psi\0\big)\inv\big)$ from $\SA(K)$ to $\SA(H)$.
	The \egmor $(\TG\cdot\SA)([\id_K,\psi])$ is nothing but the \egmor $\big[\varphi_{\xi'}, \psi_{\theta'}; \Bis K \ltimes_{\gamma^H\theta'} \psi\0\big(K\0\big)\big]$.
	We have $\varphi_{\xi'}([V,w]) = \big[V,\big(\psi\0\big)\inv(w)\big]$ and $\psi_{\theta'}([V,w]) = [\psi(V),w]$ for $[V,w] \in \Bis K \ltimes_{\gamma^H\theta'} \psi\0\big(K\0\big)$.
	In a similar way to the above discussion, we get 
	\[
	\varepsilon_H \cdot (\TG\cdot\SA)([\id_K,\psi]) \cdot \varepsilon_K\inv = [ \omega_K\circ\varphi_{\xi'}, \omega_H\circ\psi_{\theta'}].
	\]
	Let us show $[ \omega_K\circ\varphi_{\xi'}, \omega_H\circ\psi_{\theta'}] = [\id_K,\psi]$.
	The continuous groupoid homomorphism $\omega_K\circ\varphi_{\xi'}$ becomes an equivalence from $\big(\omega_K\circ\varphi_{\xi'}, \omega_H\circ\psi_{\theta'}; \Bis K \ltimes_{\gamma^H\theta'} \psi\0\big(K\0\big) \big)$ to $(\id_K,\psi;K)$;
	\[
	\begin{tikzcd}
	&\Bis K \ltimes_{\gamma^H\theta'} \psi\0\big(K\0\big) &\\
	K&&H.\\
	&K&
	\arrow[from = 1-2, to = 2-1, "{\omega_K\circ\varphi_{\xi'}}"']
	\arrow[from = 1-2, to = 2-3, "{\omega_H\circ\psi_{\theta'}}"]
	\arrow[from = 3-2, to = 2-1, "{\id_K}"]
	\arrow[from = 3-2, to = 2-3, "{\psi}"']
	\arrow[from = 1-2, to = 3-2, "{\omega_K\circ\varphi_{\xi'}}",dashed]
	\end{tikzcd}
	\]
	It is trivial that the left triangle commutes. 
	Take $[V,w] \in \Bis K \ltimes_{\gamma^H\theta'} \psi\0\big(K\0\big)$.
	We set $k := \omega_K\big(\varphi_{\xi'}([V,w]) \big) = \omega_K\big(\big[V,\big(\psi\0\big)\inv(w)\big]\big)$.
	This element $k\in K$ satisfies that $k \in V$ and $d(k) = \big(\psi\0\big)\inv(w)$.
	The element $\psi(k) \in H$ satisfies that $\psi(k) \in \psi(V)$ and $d(\psi(k)) = \psi\0(d(k)) = w$.
	This element is nothing but $\omega_H(\psi_{\theta'}([V,w])) = \omega_H([\psi(V),w])$.
	Hence we get $\omega_H\circ\psi_{\theta'} = \psi\circ (\omega_K\circ\varphi_{\xi'})$.
	One can see that $\big(\omega_K\circ\varphi_{\xi'}\big)\0$ is surjective.
	By Proposition \ref{Proposition: egmors}, $\omega_K\circ\varphi_{\xi'}$ is an equivalence.
	Thus we get $\varepsilon_H \cdot (\TG\cdot\SA)([\id_K,\psi]) \cdot \varepsilon_K\inv = [\id_K,\psi]$.
	Hence the second square commutes.
\end{proof}

This lemma means that the collection $\varepsilon = \{\varepsilon_G\}$ forms a natural isomorphism from $\TG\cdot\SA$ to $1_{\EG}$.

\begin{theo}\label{Theorem: adj func btwn ISA and EG}
	The functor $\TG$ from $\ISA$ to $\EG$ is left adjoint to the functor $\SA$ from $\EG$ to $\ISA$.
\end{theo}
\begin{proof}
	It suffices to show that the following two triangles commute (see \cite[Section IV]{Mac98});
	\[
	\begin{tikzcd}
	\TG(S,X,\alpha)&[4mm] (\TG\cdot\SA\cdot\TG)(S,X,\alpha)\\
	&\TG(S,X,\alpha),
	\arrow[r,from = 1-1, to = 1-2, "\TG(\eta_{(S,X,\alpha)})"]
	\arrow[d, from =1-2, to = 2-2, "\varepsilon_{\TG(S,X,\alpha)}"]
	\arrow[rd, from = 1-1, to = 2-2, ""',Rightarrow,no head]
	\end{tikzcd}
	\]
	\[
	\begin{tikzcd}
	\SA(G)&(\SA\cdot\TG\cdot\SA)(G)\\
	&\SA(G)
	\arrow[r,from = 1-1, to = 1-2, "\eta_{\SA(G)}"]
	\arrow[d, from =1-2, to = 2-2, "\SA(\varepsilon_G)"]
	\arrow[rd, from = 1-1, to = 2-2, ""', Rightarrow,no head]
	\end{tikzcd}
	\]
	for an inverse semigroup action $(S,X,\alpha)$ and an \'etale groupoid $G$.
	
	First, we show that the first triangle commutes.
	The \egmor 
	\[
	\varepsilon_{\TG(S,X,\alpha)} \cdot \TG(\eta_{(S,X,\alpha)}) = [\id_{(\TG\cdot\SA)(S\ltimes_\alpha X)},\omega_{S\ltimes_\alpha X}] \cdot [\varphi_{\xi_\alpha},\psi_{\theta_\alpha}]
	\]
	coincides with 
	\[
	[\varphi_{\xi_\alpha},\omega_{S\ltimes_\alpha X}\circ\psi_{\theta_\alpha};S\ltimes_{\gamma^{S\ltimes_\alpha X}\theta_\alpha}(S\ltimes_\alpha X)\0]
	\]
	by Lemma \ref{Lemma: unitor}.
	Since $\xi_\alpha$ is homeomorphic, $\varphi_{\xi_\alpha}$ is a homeomorphic groupoid isomorphism by Remark \ref{Remark: phi homeo iff phi0 homeo}.
	We can check that 
	\begin{align*}
	\varphi_{\xi_\alpha} ([s,[e,x]]) &= [s,\xi_\alpha([e,x])] = [s,x],\\
	(\omega_{S\ltimes_\alpha X}\circ\psi_{\theta_\alpha})([s,[e,x]]) 
	&= \omega_{S\ltimes_\alpha X} ([\theta_\alpha(s),[e,x]])\\
	&= \omega_{S\ltimes_\alpha X} ([[s,D^\alpha_s],[e,x]]) 
	= [s,x],
	\end{align*}
	for every $[s,[e,x]] \in S\ltimes_{\gamma^{S\ltimes_\alpha X}\theta_\alpha}(S\ltimes_\alpha X)\0$.
	Hence we get $\varphi_{\xi_\alpha} = \omega_{S\ltimes_\alpha X}\circ\psi_{\theta_\alpha}$.
	Thus $[\varphi_{\xi_\alpha},\omega_{S\ltimes_\alpha X}\circ\psi_{\theta_\alpha}]$ coincides with the identity \egmor for $\TG(S,X,\alpha) = S\ltimes_\alpha X$ by Lemma \ref{Lemma: unitor 2}.
	
	Next we show that the second triangle commutes.
	The \isamor $\SA(\varepsilon_G) \cdot \eta_{\SA(G)}$ is nothing but the \isamor
	\[
	\big( \varepsilon_G , \big(\varepsilon_G\0\big)\inv \big) \cdot \big( \theta_{\gamma^G}, \xi_{\gamma^G} \big) 
	= \big( \varepsilon_G \circ \theta_{\gamma^G}, \xi_{\gamma^G} \circ \big(\varepsilon_G\0\big)\inv \big).
	\]
	We can check that
	\begin{align*}
	\big( \varepsilon_G \circ \theta_{\gamma^G} \big)(U) &= \varepsilon_G([U,d(U)]) = U\\
	\big( \xi_{\gamma^G} \circ \big(\varepsilon_G\0\big)\inv \big)(u) &= \xi_{\gamma^G}\big( [G\0,u] \big) = u
	\end{align*}
	for every $U \in \Bis G$ and $u\in G\0$.
	Hence we get $\varepsilon_G \circ \theta_{\gamma^G} = \id_{\Bis G}$ and $\xi_{\gamma^G} \circ \big(\varepsilon_G\0\big)\inv = \id_{G\0}$.
	Thus the \isamor $\SA(\varepsilon_G) \cdot \eta_{\SA(G)}$ coincides with the identity \isamor for $\SA(G) = \big(\Bis G,G\0,\gamma^G\big)$.
\end{proof}

\begin{coro}
	The functor $\TG$ from $\ISAlcH$ to $\EGlcH$ is left adjoint to the functor $\SA$ from $\EGlcH$ to $\ISAlcH$.
\end{coro}
\begin{proof}
	We can get this corollary since the \isamors $\eta_{(S,X,\alpha)}$ is proper for all inverse semigroup actions $(S,X,\alpha)$, and the \egmors $\varepsilon_G$ is proper for all \'etale groupoids.
\end{proof}


\section{A functor $C^*$ from $\EGlcH$ to $\Calg$}
\label{Section: A functor C* from EG to Calg}

We recall that $\Calg$ is the category of all $C^*$-algebras and all \shoms.
We construct a functor $C^*$ from $\EGlcH$ to $\Calg$ in Theorem \ref{Theorem: functor C*}.
Let $G,K$ and $H$ be \'etale groupoids with locally compact Hausdorff unit spaces.

\begin{prop}\label{Proposition: sigma_varphi}
	Let $\varphi\: K \rightarrow G$ be a continuous fibrewise bijective groupoid homomorphism with proper $\varphi\0$.  
	For every function $f\: G \rightarrow \mathbb{C}$, we define a function $\sigma_\varphi(f)\: K \rightarrow \mathbb{C}$ as
	\[
	\sigma_\varphi(f) := f\circ \varphi.
	\]
	The map $\sigma_\varphi: Q(G) \rightarrow Q(K); f\mapsto \sigma_\varphi(f)$ is a well-defined \shom. 
\end{prop}
\begin{proof}
	Let $U$ be an open bisection of $G$ and $f$ be an element of $C_c(U) \subset Q(G)$.
	Then there exists a compact subset $V$ of $U$ such that $f$ is zero outside of $V$.
	The set $\varphi\inv(V)$ is a compact subspace of an open bisection $\varphi\inv(U)$ of $K$ by Lemma \ref{Lemma: varphi inv} (iv) and Lemma \ref{Lemma: phi proper}.
	The function $\sigma_\varphi(f)$ is zero outside of $\varphi\inv(V)$, and is a function in $C_c(\varphi\inv(U)) \subset Q(K)$.
	The map $f \mapsto \sigma_\varphi(f)$ is linear as a map between linear spaces of all complex functions on $G$ and $K$.
	Thus $\sigma_\varphi\: Q(G) \rightarrow Q(K)$ is well-defined.
	
	Let $U_i$ be an open bisection of $G$, and $f_i$ be an element of $C_c(U_i) \subset Q(G)$ for $i=1,2$.
	By Proposition \ref{Proposition: convolution} (i), we get $f_1*f_2 \in C_c(U_1U_2)$.
	Thus $\sigma_\varphi(f_1*f_2) \in C_c(\varphi\inv(U_1U_2))$ holds by the above discussion.
	Proposition \ref{Proposition: convolution} (i) also implies that $\sigma_\varphi(f_1) * \sigma_\varphi(f_2) \in C_c(\varphi\inv(U_1)\varphi\inv(U_2))$.
	By Lemma \ref{Lemma: varphi inv} (i), $\varphi\inv(U_1)\varphi\inv(U_2) = \varphi\inv(U_1U_2)$ holds.
	Thus $\sigma_\varphi(f_1*f_2)$ and $\sigma_\varphi(f_1) * \sigma_\varphi(f_2)$ belong to the same subspace $C_c(\varphi\inv(U_1)\varphi\inv(U_2))$.
	For $k_1\in \varphi\inv(U_1)$ and $k_2\in \varphi\inv(U_2)$ with $d(k_1) = r(k_2)$, we get 
	\begin{align*}
	\big(\sigma_\varphi(f_1)*\sigma_\varphi(f_2)\big)(k_1k_2)
	&= \sigma_\varphi(f_1)(k_1)\sigma_\varphi(f_2)(k_2)\\
	&= f_1(\varphi(k_1))f_2(\varphi(k_2))\\
	&= (f_1*f_2)(\varphi(k_1)\varphi(k_2))\\
	&= (f_1*f_2)(\varphi(k_1k_2))\\
	&= \sigma_\varphi(f_1*f_2)(k_1k_2) 
	\end{align*}
	by Proposition \ref{Proposition: convolution} (i).
	Thus $\sigma_\varphi(f_1*f_2) = \sigma_\varphi(f_1) * \sigma_\varphi(f_2)$ holds.
	By linearity, this equation also holds for every $f_1,f_2 \in Q(K)$.
	
	For every $f \in Q(G)$ and $k \in K$, we get  
	\begin{align*}
		\sigma_\varphi(f)^*(k)
		= \overline{\sigma_\varphi(f)(k\inv)} 
		&= \overline{f(\varphi(k\inv))}\\
		&= \overline{f(\varphi(k)\inv)} 
		= f^*(\varphi(k))
		= \sigma_\varphi(f^*)(k).
	\end{align*}
	Thus $\sigma_\varphi(f)^* = \sigma_\varphi(f^*)$ holds.
	We finish to show that $\sigma_\varphi$ is a \shom from $Q(G)$ to $Q(K)$.
\end{proof}

\begin{defi}
	We use the same symbol $\sigma_\varphi\: C^*(G) \rightarrow C^*(K)$ for denoting the extension of the \shom $\sigma_\varphi\: Q(G) \rightarrow Q(K) \subset C^*(K)$.
\end{defi}

\begin{prop}\label{Proposition: pi_psi}
	Let $\psi\: K\rightarrow H$ be a continuous groupoid homomorphism with open injective $\psi\0$.
	For a function $f \in Q(K)$, we define a function $\pi_\psi(f) \in Q(H)$ as
	\[
	\pi_\psi(f)(h) := \sum_{\substack{k\in K\\\psi(k) = h}} f(k)
	\]
	for every $h\in H$.
	This map $\pi_\psi\: Q(K) \rightarrow Q(H); f \mapsto \pi_\psi(f)$ becomes a well-defined \shom.
\end{prop}
\begin{proof}
	Let $U$ be an open bisection of $K$ and $f$ be an element of $C_c(U) \subset Q(K)$.
	By Lemma \ref{Lemma: psi 2}, $\psi|_U\:U\rightarrow\psi(U)$ is a homeomorphism.
	The function $f \circ \psi|_U\inv$ is an element of $C_c(\psi(U))$.
	Since $f$ is zero outside of $U$, we get
	\[
	\pi_\psi(f)(h) = \sum_{\substack{k\in K\\\psi(k) = h}} f(k) = \sum_{\substack{k\in U\\\psi(k) = h}} f(k) = 
	\begin{cases}
	\big(f\circ\psi|_U\inv\big)(h) & h \in \psi(U),\\
	0 & h \not\in \psi(U).
	\end{cases}
	\]
	Thus this function $\pi_\psi(f)$ belongs to $C_c(\psi(U)) \subset Q(H)$.
	The above discussion implies that for every $f \in Q(K)$ the function $\pi_\psi(f)$ exists and it belongs to $Q(H)$.
	
	It is easy to check that $\pi_\psi\: Q(K) \rightarrow Q(H)$ is linear.
	We show that $\pi_\psi$ keeps multiplication and involution.
	
	Take an open bisection $U_i$ of $K$ and a function $f_i\in C_c(U_i) \subset Q(K)$ for $i=1,2$.
	By Proposition \ref{Proposition: convolution} (i), we get $f_1*f_2\in C_c(U_1U_2)$. 
	Thus $\pi_\psi(f_1*f_2) \in C_c(\psi(U_1U_2))$ holds by the above discussion.
	Proposition \ref{Proposition: convolution} (i) also implies that $\pi_\psi (f_1) * \pi_\psi(f_2) \in C_c(\psi(U_1)\psi(U_2))$.
	By Lemma \ref{Lemma: psi}, we get $\psi(U_1U_2) = \psi(U_1)\psi(U_2)$.
	Thus $\pi_\psi(f_1*f_2)$ and $\pi_\psi (f_1) * \pi_\psi(f_2)$ belong to the same subspace $C_c(\psi(U_1U_2))$.
	Take $k_1\in U_1$ and $k_2\in U_2$ with $d(k_1) = r(k_2)$.
	We get
	\begin{align*}
	\big(\pi_\psi(f_1)*\pi_\psi(f_2)\big)(\psi(k_1k_2))
	&= \big(\pi_\psi(f_1)*\pi_\psi(f_2)\big)(\psi(k_1)\psi(k_2))\\
	&= \pi_\psi(f_1)(\psi(k_1)) \pi_\psi(f_2)(\psi(k_2))\\
	&= f_1\big(\psi|_{U_1}\inv(\psi(k_1))\big)f_2\big(\psi|_{U_2}\inv(\psi(k_2))\big)\\
	&= f_1(k_1)f_2(k_2)\\
	&= (f_1*f_2)(k_1k_2)\\
	&= (f_1*f_2)(\psi|_{U_1U_2}\inv(\psi(k_1k_2)))\\
	&= \pi_\psi(f_1*f_2)(\psi(k_1k_2))
	\end{align*}
	by Proposition \ref{Proposition: convolution} (i).
	Thus we have $\pi_\psi(f_1) * \pi_\psi(f_2) = \pi_\psi(f_1 * f_2)$ holds.
	By linearity, this equation also holds for every $f_1,f_2\in Q(K)$.
	
	Take an open bisection $U$ of $K$ and a function $f\in C_c(U) \subset Q(K)$.
	By Proposition \ref{Proposition: convolution} (ii), we get $f^*\in C_c(U\inv)$.
	Thus $\pi_\psi(f^*) \in C_c(\psi(U\inv))$ holds by the above discussion.
	Proposition \ref{Proposition: convolution} (ii) also implies that $\pi_\psi(f)^* \in C_c(\psi(U)\inv)$ holds.
	By Lemma \ref{Lemma: grpd hom and inv}, we get $\psi(U\inv) = \psi(U)\inv$. 
	Thus $\pi_\psi(f^*)$ and $\pi_\psi(f)^*$ belong to the same subspace $C_c(\psi(U)\inv)$.
	For $k\in U$, we get 
	\begin{align*}
	\pi_\psi(f)^*(\psi(k)\inv)
	&= \overline{\pi_\psi(f)(\psi(k))}\\
	&= \overline{f(\psi|_{U}\inv(\psi(k)))}\\
	&= \overline{f(k)}\\
	&= f^*(k\inv)\\
	&= f^*\big(\psi|_{U\inv}\inv(\psi(k\inv))\big)\\
	&= \pi_\psi(f^*)(\psi(k\inv))\\
	&= \pi_\psi(f^*)(\psi(k)\inv).
	\end{align*} 
	Thus we have $\pi_\psi(f)^* = \pi_\psi(f^*)$.
	By linearity, this equation also holds for every function $f \in Q(K)$.
	We have shown that $\pi_\psi$ is a \shom from $Q(K)$ to $Q(H)$.
\end{proof}

\begin{defi}
	We use the same symbol $\pi_\psi\: C^*(K) \rightarrow C^*(H)$ for denoting the extension of the \shom $\pi_\psi\: Q(K) \rightarrow Q(H) \subset C^*(H)$.
\end{defi}

Let $G_i$ be an \'etale groupoid with a locally compact Hausdorff unit space for $i = 1,2,3$.
\begin{lemm}\label{Lemma: sigma composition}
	Let $\varphi_i\: G_i \rightarrow G_{i+1}$ be a continuous fibrewise bijective groupoid homomorphism with proper $\varphi_i\0$ for $i=1,2$.
	Two \shoms $\sigma_{\varphi_2\circ\varphi_1}$ and $\sigma_{\varphi_1}\circ\sigma_{\varphi_2}$ from $C^*(G_3)$ to $C^*(G_1)$ coincide.
\end{lemm}
\begin{proof}
	For every $f \in Q(G_3)$,
	\[
	(\sigma_{\varphi_1}\circ\sigma_{\varphi_2})(f) 
	= \sigma_{\varphi_1}(f \circ \varphi_2) 
	= f \circ \varphi_2 \circ \varphi_1 
	= \sigma_{\varphi_2\circ\varphi_1} (f)
	\]
	holds.
	Thus $(\sigma_{\varphi_1}\circ\sigma_{\varphi_2})(x) = \sigma_{\varphi_2\circ\varphi_1} (x)$ holds for all $x \in C^*(G_3)$.
\end{proof}

\begin{lemm}\label{Lemma: pi composition}
	Let $\psi_i\: G_i\rightarrow G_{i+1}$ be a continuous groupoid homomorphism with open injective $\psi_i\0$ for $i=1,2$.
	Two \shoms $\pi_{\psi_2\circ\psi_1}$ and $\pi_{\psi_2} \circ \pi_{\psi_1}$ from $C^*(G_1)$ to $C^*(G_3)$ coincide.
\end{lemm}
\begin{proof}
	For every open bisection $U$ of $G_1$ and $f \in C_c(U) \subset Q(G_1)$, the functions $\pi_{\psi_2\circ\psi_1}(f)$ and $(\pi_{\psi_2} \circ \pi_{\psi_1})(f)$ lie in $C_c(\psi_2(\psi_1(U))) \subset Q(G_3)$.
	We get 
	\begin{align*}
	(\pi_{\psi_2} \circ \pi_{\psi_1})(f)|_{\psi_2(\psi_1(U))}
	&= \pi_{\psi_1}(f) \circ \psi_2|_{\psi_1(U)}\inv\\
	&= f \circ \psi_1|_{U}\inv \circ \psi_2|_{\psi_1(U)}\inv\\
	&= f \circ (\psi_2\circ\psi_1)|_U\inv\\
	&= \pi_{\psi_2\circ\psi_1}(f)|_{\psi_2(\psi_1(U))}.
	\end{align*}
	Thus we get $\pi_{\psi_2\circ\psi_1}(f) = (\pi_{\psi_2} \circ \pi_{\psi_1})(f)$.
	By linearity, the same equation holds for every $f \in Q(G_1)$.
	Thus $(\pi_{\psi_2} \circ \pi_{\psi_1})(x) = \pi_{\psi_2\circ\psi_1}(x)$ holds for all $x \in C^*(G_1)$.
\end{proof}

\begin{lemm}\label{Lemma: sigma_iota, pi_iota}
	Let $K_1$, $K_2$ be \'etale groupoids with locally compact Hausdorff unit spaces, and $\iota\: K_1 \rightarrow K_2$ be a homeomorphic groupoid isomorphism.
	The $*$-homomorphism $\pi_\iota\circ\sigma_\iota$ from $C^*(K_2)$ to $C^*(K_2)$ coincides with the identity $*$-homomorphism for $C^*(K_2)$.
\end{lemm}
\begin{proof}
	It is easy to prove this lemma.
\end{proof}

\begin{prop}\label{Proposition: well-defined of C*}
	For \'etale groupoids $G,H$ with locally compact Hausdorff unit spaces and equivalent proper \egmors $(\varphi_1,\psi_1;K_1)$ and $(\varphi_2,\psi_2;K_2)$ from $G$ to $H$, two \shoms $\pi_{\psi_1}\circ\sigma_{\varphi_1}$ and $\pi_{\psi_2}\circ\sigma_{\varphi_2}$ from $C^*(G)$ to $C^*(H)$ coincide.
\end{prop}
\begin{proof}
	Let $\iota\: K_1\rightarrow K_2$ be a homeomorphic groupoid isomorphism with $\varphi_2\circ\iota = \varphi_1$ and $\psi_2\circ\iota = \psi_1$.
	Lemma \ref{Lemma: sigma composition}, \ref{Lemma: pi composition}, and \ref{Lemma: sigma_iota, pi_iota} imply that
	\begin{equation*}
	\pi_{\psi_1} \circ \sigma_{\varphi_1} 
	= \pi_{\psi_2 \circ \iota} \circ \sigma_{\varphi_2 \circ \iota} 
	= \pi_{\psi_2} \circ \pi_\iota \circ \sigma_\iota \circ \sigma_{\varphi_2}
	= \pi_{\psi_2} \circ \sigma_{\varphi_2}.\qedhere
	\end{equation*}
\end{proof}

\begin{theo}\label{Theorem: functor C*}
	The constructions $G\mapsto C^*(G)$ and $[\varphi,\psi] \mapsto \pi_\psi\circ\sigma_\varphi$ form a functor from $\EGlcH$ to $\Calg$.
	We denote this functor as $C^*$.
\end{theo}
\begin{proof}
	The map $[\varphi,\psi] \mapsto \pi_\psi\circ\sigma_\varphi$ is well-defined by Proposition \ref{Proposition: well-defined of C*}.
	
	Let $G_i$ be an \'etale groupoid with a locally compact Hausdorff unit space with $i=1,2,3$, and $(\varphi_i,\psi_i;K_i)\: G_i \rightarrow G_{i+1}$ be a proper \egmor with $i = 1,2$.
	Consider the \egmor $\big(\widetilde{\varphi_2},\widetilde{\psi_1};K_1\times_{G_2} K_2\big)$ from $K_1$ to $K_2$ as in Proposition \ref{Proposition: comp cpl mor}. 
	Take an open bisection $U$ of $K_1$ and $f \in C_c(U) \subset Q(K_1)$.
	We set the open bisection $V:=\varphi_2\inv(\psi_1(U)) = \widetilde{\psi_1}(\widetilde{\varphi_2}\inv(U))$ of $K_2$ (see Lemma \ref{Lemma: pull-back 1}).
	The functions $\big(\pi_{\widetilde{\psi_1}} \circ \sigma_{\widetilde{\varphi_2}}\big)(f)$ and $\big(\sigma_{\varphi_2} \circ \pi_{\psi_1}\big)(f)$ lies in $C_c(V) \subset Q(K_2)$.
	We get
	\begin{align*}
		\big(\pi_{\widetilde{\psi_1}} \circ \sigma_{\widetilde{\varphi_2}}\big)(f)|_V
		&= f \circ \widetilde{\varphi_2} \circ \widetilde{\psi_1}|\inv_{\widetilde{\varphi_2}\inv(U)}\\
		&= f \circ \left.\psi_1\right|_U\inv \circ \varphi_2 
		= \big(\sigma_{\varphi_2} \circ \pi_{\psi_1}\big)(f)|_V
	\end{align*}
	by Lemma \ref{Lemma: pull-back 2}.
	Hence \shoms $\pi_{\widetilde{\psi_1}} \circ \sigma_{\widetilde{\varphi_2}}$ and $\sigma_{\varphi_2} \circ \pi_{\psi_1}$ from $C^*(K_1)$ to $C^*(K_2)$ coincide. 
	By Lemma \ref{Lemma: sigma composition} and \ref{Lemma: pi composition}, we get
	\[
	\pi_{\psi_2\circ\widetilde{\psi_1}} \circ \sigma_{\varphi_1\circ\widetilde{\varphi_2}} = \pi_{\psi_2} \circ \pi_{\widetilde{\psi_1}} \circ \sigma_{\widetilde{\varphi_2}} \circ \sigma_{\varphi_1} = \pi_{\psi_2} \circ \sigma_{\varphi_2} \circ \pi_{\psi_1} \circ \sigma_{\varphi_1}.
	\]
	Thus the construction $[\varphi,\psi] \mapsto \pi_\psi\circ\sigma_\varphi$ preserves composition. 
	
	Let $G$ be an \'etale groupoid with a locally compact Hausdorff unit space.
	The identity \egmor $[\id_G,\id_G;G]$ on $G$ produces the identity \shom on $C^*(G)$ by Lemma \ref{Lemma: sigma_iota, pi_iota}.
	Thus the construction $[\varphi,\psi] \mapsto \pi_\psi\circ\sigma_\varphi$ preserves identity morphisms. 
\end{proof}

\begin{prop}\label{Proposition: faithfulness of C*}
	The functor $C^*$ from $\EGlcH$ to $\Calg$ is faithful.
\end{prop}
\begin{proof}
	Let $[\varphi_1,\psi_1;K_1]$ and $[\varphi_2,\psi_2;K_2]$ be \egmors from $G$ to $H$ with $[\varphi_1,\psi_1;K_1] \neq [\varphi_2,\psi_2;K_2]$.
	By Corollary \ref{Corollary: egmor equiv}, this implies that $\psi_1(\varphi_1\inv(U)) \neq \psi_2(\varphi_2\inv(U))$ holds for some open bisection $U$ of $G$.
	The proof finishes if we show that the two \shoms $\pi_{\psi_1} \circ \sigma_{\varphi_1}$ and $\pi_{\psi_2} \circ \sigma_{\varphi_2}$ from $Q(G)$ to $Q(H)$ do not coincide.
	
	Without loss of generality, we can take an element $g\in \varphi_1\inv(U)$ with $\psi_1(g)\not\in \psi_2(\varphi_2\inv(U))$.
	There exists an element $f \in C_c(U) \subset Q(G)$ with $f(\varphi_1(g)) = 1$ by Urysohn's Lemma (see \cite[Lemma 2.12]{Rud87}).
	We get $\big(\pi_{\psi_1}\circ\sigma_{\varphi_1}\big)(f)(\psi_1(g)) = 1$.
	However $\big(\pi_{\psi_2}\circ\sigma_{\varphi_2}\big)(f)(\psi_1(g)) = 0$ holds because $\big(\pi_{\psi_2}\circ\sigma_{\varphi_2}\big)(f)$ is zero outside of $\psi_2(\varphi_2\inv(U))$.
	Thus the two \shoms $\pi_{\psi_1}\circ\sigma_{\varphi_1}$ and $\pi_{\psi_2}\circ\sigma_{\varphi_2}$ from $Q(G)$ to $Q(H)$ do not coincide.
\end{proof}

Now we proceed to show our categorical version of \cite[Theorem 4.1.1]{Pat99}. 
Let $G$ be an \'etale groupoid with a locally compact Hausdorff unit space.
For a compact open bisection $C$ of $G$, the characteristic function $\chi_C$ of $C$ is an element of $Q(G)$.
\begin{lemm}\label{Lemma: Paterson}
	Let $G$, $H$ be \'etale groupoids with locally compact Hausdorff unit spaces, and $[\varphi,\psi;K]$ be a proper \egmor from $G$ to $H$.
	For a compact open bisection $C$ of $G$, we have 
	\[
	\pi_\psi \big(\sigma_\varphi ( \chi_C )\big) = \chi_{\psi(\varphi\inv(C))}.
	\] 
\end{lemm}
\begin{proof}
	By Lemma \ref{Lemma: phi proper}, $\varphi \: \varphi\inv(C) \rightarrow C$ is proper.
	Thus $C' := \varphi\inv(C)$ is a compact open bisection of $K$.
	
	For every $k\in K$, we get 
	\[
	\sigma_\varphi(\chi_C)(k) = \chi_C(\varphi(k)) = \chi_{\varphi\inv(C)}(k).
	\]
	Hence $\sigma_\varphi(\chi_C) = \chi_{C'}$.
	The function $\chi_{C'}$ lies in $C_c(C') \subset Q(K)$.
	For every $h\in H$, we get 
	\[
	\pi_\psi(\chi_{C'})(h) = 
	\begin{cases}
		\chi_{C'}\big(\psi|_{C'}\inv(h)\big) = 1 & h\in \psi(C')\\
		0 & h\not\in \psi(C')
	\end{cases}
	\]
	(see Proposition \ref{Proposition: pi_psi}).
	Thus we get $\pi_\psi(\chi_{C'}) = \chi_{\psi(C')}$, that is, $\pi_\psi \big(\sigma_\varphi ( \chi_C )\big) = \chi_{\psi(\varphi\inv(C))}$.
\end{proof}

\begin{theo}\label{Theorem: Paterson categorical ver}
	The composition of the functors $\SP\: \IS \rightarrow \ISAlcH$, $\TG\: \ISAlcH \rightarrow \EGlcH$, and $C^*\: \EGlcH \rightarrow \Calg$ is naturally equivalent to the functor $C^*\: \IS \rightarrow \Calg$ through the natural isomorphism consisting of the $*$-isomorphisms $\iota_S \: C^*(S) \rightarrow C^*(\Gu(S))$ for all inverse semigroups $S$ in Theorem \ref{Theorem: Paterson};
	\[
	\begin{tikzcd}
	\IS&&&\Calg.\\
	&\ISAlcH&\EGlcH&
	\arrow[from = 1-1, to = 1-4, rightarrow, ""'{name = a}, "C^*"]
	\arrow[from = 1-1, to = 2-2, rightarrow, "\SP"']
	\arrow[from = 2-2, to = 2-3, rightarrow, ""{name = b}, "\TG"']
	\arrow[from = 2-3, to = 1-4, rightarrow, "C^*"']
	\arrow[from = a, to = b, "\simeq", phantom, sloped]
	\end{tikzcd}
	\]
	%
\end{theo}
\begin{proof}
	Let $S,T$ be inverse semigroups, and $\theta: S \rightarrow T$ be a semigroup homomorphism.
	We show the following diagram commutes;
	\[
	\begin{tikzcd}
		C^*(S) \arrow[r,"\iota_S"] \arrow[d,"\sigma_\theta"']& 
		C^*(\Gu(S)) \arrow[d,"\pi_{\psi_\theta}\circ\sigma_{\varphi_{\widehat{\theta}}}"]\\
		C^*(T) \arrow[r,"\iota_T"]& 
		C^*(\Gu(T)),
	\end{tikzcd}
	\]
	where $\Gu$ is the functor defined in Definition \ref{Definition: functor Gu}.
	For every $s\in S$, we have
	\begin{align*}
	(\pi_{\psi_\theta}\circ\sigma_{\varphi_{\widehat{\theta}}})\big( \iota_S(\delta_s) \big)
	&= (\pi_{\psi_\theta}\circ\sigma_{\varphi_{\widehat{\theta}}})\big( \chi_{[s,U^S_s]}\big)\\ 
	&= \chi_{\psi_\theta\big(\varphi_{\widehat{\theta}}\inv \big(\big[s,U^S_s\big]\big)\big)},\\
	\iota_T\big(\sigma_\theta(\delta_s)\big)
	&= \iota_T\big(\delta_{\theta(s)}\big)\\
	&= \chi_{[\theta(s),U^T_{\theta(s)}]},
	\end{align*} 
	where the second equal follows from Lemma \ref{Lemma: Paterson}.
	These functions coincide by the following calculation;
	\[
	\psi_\theta\big(\varphi_{\widehat{\theta}}\inv \big(\big[s,U^S_s\big]\big)\big)
	= \psi_\theta\big(\big[s,\widehat{\theta}\inv\big(U^S_s\big)\big]\big)
	= \psi_\theta\big(\big[s,U^T_{\theta(s)}\big]\big)
	= \big[\theta(s), U^T_{\theta(s)}\big]. \qedhere
	\]
\end{proof}

\bibliography{reference}
\bibliographystyle{alpha}
\end{document}